\documentclass[11pt,leqno]{article}
\usepackage{calc}
\usepackage{a4wide}
\usepackage[ngerman,english]{babel}
\usepackage{etex}
\usepackage[T1]{fontenc}
\usepackage[utf8]{inputenc}
\usepackage{ellipsis,ragged2e}
\usepackage{microtype}
\usepackage{amsmath}
\usepackage{amssymb}
\usepackage{amsthm}
\usepackage{color}
\definecolor{rasp}{rgb}{.89,.04,.36}      % rot fuer mid

\usepackage{times}
\usepackage{pictexwd}
\usepackage[pdftex]{graphicx}
\graphicspath{{figures/}}
\usepackage{caption}
\usepackage{subcaption}
\usepackage{pst-pdf,pstricks-add}
\usepackage{hyperref}
\usepackage{tikz}
\usetikzlibrary{plotmarks}
\usepackage{pgfplots}
\pgfplotsset{compat=newest}
\usepackage{dsfont}
\usepackage{kpfonts}
\usepackage{listings}
\usepackage{cleveref}
\usepackage{cite}

\usepackage{mathrsfs,float,indentfirst,textcomp}
\usepackage{setspace}
\usepackage{latexsym,lscape,amsfonts,cite,enumerate}
\usepackage{lscape}
\usepackage{xfrac}
\usepackage{algorithm} 
\usepackage{algorithmic}
\usepackage{multirow} 
\usepackage{xcolor}
\usepackage{geometry}
\usepackage{booktabs}

\newtheorem{theorem}{Theorem}[section]

\newtheorem{proposition}[theorem]{Proposition}
\newtheorem{mydef}[theorem]{Definition}
\newtheorem{remark}{Remark}

\newcommand{\R}{\mathbb{R}}
\newcommand{\G}{\mathcal{G}}

\newcommand{\vecsym}[1]{\boldsymbol{#1}}
\renewcommand{\vec}[1]{\mathbf{#1}}

\newcommand{\revision}[1]{\textcolor{black}{#1}}
\newcommand{\rrevision}[1]{\textcolor{black}{#1}}

\title{A Stabilization of a Continuous Limit of the Ensemble Kalman Inversion}
\date{\today}
\author{
	Dieter Armbruster \medskip\\
	{\small\it School of Mathematical and Statistical Sciences} \\
	{\small\it Arizona State University} \\
	{\small\it Tempe, AZ 85287-1804, USA
	}
	\bigskip \\
	Michael Herty \medskip\\
	{\small\it Institut f\"{u}r Geometrie und Praktische Mathematik (IGPM)} \\
	{\small\it RWTH Aachen University} \\
	{\small\it Templergraben 55, 52062 Aachen, Germany
	}
	\bigskip \\
	Giuseppe Visconti \medskip\\
	{\small\it Dipartimento di Matematica ``G. Castelnuovo''} \\
	{\small\it La Sapienza, Universit\`{a} di Roma} \\
	{\small\it P.le Aldo Moro 5, 00185 Roma, Italy
	}
}

\begin{document}

\maketitle

% REQUIRED
\begin{abstract}
  The Ensemble Kalman Filter (EnKF) belongs to the class of iterative particle filtering methods and can be used for solving control--to--observable inverse problems. In this context, the EnKF is known as Ensemble Kalman Inversion (EKI). In recent years several continuous limits in the number of iteration and particles have been performed in order to study properties of the method. In particular, a one--dimensional linear stability analysis reveals
  %a possible instability of the solution
  possible drawbacks in the phase space of moments
  provided by the continuous limits of the EKI, but observed also in the multi--dimensional setting.
  In this work we address this issue by introducing a \emph{stabilization} of the dynamics which leads to a method with globally asymptotically stable solutions. We illustrate the performance of the stabilized version by using test inverse problems from the literature and comparing it with the classical continuous limit formulation of the method.
\end{abstract}

\paragraph{Mathematics Subject Classification (2010)}  Dynamical systems, inverse problems, regularization, stabilization, nonlinear filtering methods, moment equations

\paragraph{Keywords}  37N35 (Dynamical systems in control), 65N21 (Inverse problems), 93E11 (Filtering)

\section{Introduction} \label{sec:introduction}

In this paper we investigate a particular numerical method for solving inverse problems, namely, the Ensemble Kalman Inversion (EKI), originally introduced in~\cite{iglesiaslawstuart2013}. This method can be derived in the framework of the Ensemble Kalman Filter (EnKF) as briefly explained later in this introduction and in Section~\ref{sec:preliminary}. While the EnKF has already been introduced more than ten years ago~\cite{Evensen1994,EmerickReynolds,ChenOliver,BergemannReich} as a discrete time method to estimate state variables and parameters of stochastic dynamical systems, the EKI has been recently and successfully applied to solve inverse problems in many research fields due to its derivative--free structure, in particular in oceanography~\cite{EvensenVanLeeuwen1996}, reservoir modeling~\cite{Aanonsen2009}, weather forecasting~\cite{McLaughlin2014}, milling process~\cite{SchwenzerViscontiEtAl2019}, process control~\cite{Teixeira2010}, and also machine learning~\cite{Haber2018NeverLB,KovachkiStuart2019}.

In order to set up the mathematical formulation, we let $\G:X \to Y$ be the given (possible nonlinear) forward operator between \revision{the Euclidean} spaces $X=\R^d$, $d\in\mathbb{N}$, and $Y=\R^K$, $K\in\mathbb{N}$. We are concerned with the following abstract inverse problem or parameter identification problem
\begin{equation} \label{eq:noisyProb}
	\vec{y} = \G(\vec{u}) + \vecsym{\eta}
\end{equation}
aiming to recover unknown control $\vec{u}\in X$ from given observations $\vec{y}\in Y$, where $\vecsym{\eta}$ is observational noise. Typically, $d\gg K$ and $\vecsym{\eta}$ is not explicitly known but only information on its distribution is available. We assume that $\vecsym{\eta} \sim \mathcal{N}(\vec{0},\vecsym{\Gamma}^{-1})$, i.e.~the observational noise is normally distributed with zero mean and given covariance matrix $\vecsym{\Gamma}^{-1}\in\R^{K\times K}$.

Relying on the same machinery leading to the EnKF formulation, the EKI can be derived within the inverse problem framework by rewriting~\eqref{eq:noisyProb} as a partially observed and artificial dynamical system based on state augmentation, e.g.~cf.~\cite{Anderson2001,iglesiaslawstuart2013}. The update formula for each ensemble member is computed by imposing first order necessary optimality conditions to solve a regularized minimization problem, which aims for a compromise between the background estimate of the dynamics model and additional information provided by data model. A similar technique is used to derive the update formula for constrained inverse problems~\cite{Stuart2019CEnKF,HertyVisconti2020}.

In order to understand how and why the EKI works, a continuous--time limit~\cite{schillingsetal2018,SchillingsPreprint,ChadaStuartTong2019,schillingsstuart2017,schillingsstuart2018} and a mean--field limit on the number of the ensemble members~\cite{CarrilloVaes,DingLi2019,Stuart2019MFEnKF,HertyVisconti2019} have been developed. Continuous--limits have been performed also for variants of the EKI, e.g.~we refer to~\cite{Chada2020} for the hierarchical EKI. Recent theoretical progress~\cite{schillingsstuart2017,HertyVisconti2019} using these limits \revision{in the linear setting} is the starting point of the current work. Specifically, it has been shown that, within these limits \revision{and assuming a linear forward model}, the EKI provides a solution to the inverse problem~\eqref{eq:noisyProb} by minimizing the least--squares functional
\begin{equation} \label{eq:leastSqFnc}
	\Phi(\vec{u},\vec{y}) := \frac12 \left\| \vecsym{\Gamma}^{\frac12} (\vec{y} - \G(\vec{u})) \right\|^2_Y
\end{equation}
via a preconditioned gradient flow equation, where the preconditioner is given by the empirical covariance of the ensemble, in the continuous--time limit, and via a Vlasov--type equation in the mean--field limit. Note that, contrary to the fully discrete and classical formulation of the EKI, there is no regularization of the control $\vec{u}$ in the minimization of~\eqref{eq:leastSqFnc}. However, when the inverse problem is ill-posed, infimization of $\Phi$ is not a well--posed problem and some form of regularization may be required. This has been recognized in~\cite{ChadaStuartTong2019,ZhangStroferXiao2019} where modifications of the EKI are proposed, leading to Tikhonov--Phillips--like regularizations of~\eqref{eq:leastSqFnc}. 

Another source of problems is given by the preconditioned gradient flow structure \revision{and can be analyzed with the system of the first and second central moment, i.e.~the expected value and the variance of the distribution of the ensemble, respectively.} In~\cite{HertyVisconti2019} a one--dimensional linear stability analysis of the moment equations resulting from the mean--field limit revealed that the method has infinitely many non--hyperbolic Bogdanov--Takens equilibria~\cite{guckenheimer2013nonlinear}. In this work, we \rrevision{show also that} this structure of the phase space is kept in the multi--dimensional setting. Although it is possible to show that the equilibrium providing the minimization of~\eqref{eq:leastSqFnc} is still a global attractor, having infinitely many non--hyperbolic equilibria leads to several undesirable consequences. In fact, not only are Bogdanov--Takens equilibria not asymptotically stable they are structurally unstable and thus non robust and extremely sensitive to model perturbations. These equilibria lie on the set where the preconditioner collapse to zero. Thus, convergence to the solution of the minimization of~\eqref{eq:leastSqFnc} may be affected by numerical instability which may occur, e.g., due to an overly confident prior, i.e.~when the initial ensemble is characterized by a small variance. Numerical and practical instability may either push the trajectory in the unfeasible region of the phase space, i.e.~where the variance of the ensemble is negative, or get the method stuck in the wrong equilibrium. These observations would therefore lead to the necessity of employing a proper and robust numerical discretization of the method. Furthermore, the presence of non--hyperbolic equilibria leads to a slow convergence of the method to the equilibrium solution, which happens at the rate $\mathcal{O}(t^{-1})$ where $t$ represents the time.

In this work we address these issues by introducing a modification of the continuous dynamics for the ensemble, in such a way \revision{that} the corresponding phase plane of the moment equations is characterized by a globally asymptotically stable equilibrium, the one minimizing the least--squares functional~\eqref{eq:leastSqFnc}. The stabilization effect is obtained by artificially inflating the preconditioner of the gradient flow equation\rrevision{, i.e.~the ensemble covariance matrix,} with an additive term which acts as a regularization. \rrevision{The inflation term has to be a full rank symmetric matrix.} The advantage is twofold. First, the new phase space of moments is robust. Instead, in the limits of the classical EKI, a small perturbation into the unfeasible region explodes, bringing the trajectory far from the desired one. Second, as consequence of the stabilization, the rate of convergence to the solution of the minimization of~\eqref{eq:leastSqFnc} is improved, resulting now exponentially fast. In addition, we consider a suitable acceleration/relaxation term aiming to control the distance of each ensemble member to their mean. This term further speeds up the convergence rate. The usual properties of the classical ensemble Kalman inversion, such as decay of the ensemble spread, are still satisfied by this stabilized version of the method. Its performance is investigated for an inverse problem based on a two--dimensional elliptic partial differential equation~(PDE). We show that the new method is able to converge to the solution faster and, more importantly, converges independently of the properties of the initial ensemble. \revision{Although the analysis focuses on linear forward models, numerical examples seem to suggest that the stabilization we propose provides advantages also when applied to nonlinear inverse problems. For further evidences on why EKI formulations perform well in nonlinear settings we refer to~\cite{DingLiLu}.}

We point out that, although the instability discussed above is not observed in the discrete version of the EKI, for practical purposes it is still crucial to introduce a stabilization of the continuous limits. In fact, \revision{the mean--field limit can be particularly} useful in applications because \revision{it allows to describe the case of} infinitely many ensemble members and to guarantee a computational gain in the numerical simulations using fast techniques, e.g.~see~\cite{Stuart2019MFEnKF,HertyVisconti2019,SchwenzerViscontiEtAl2019}.

The rest of the paper is organized as follows. In Section~\ref{sec:preliminary} we review the ensemble Kalman filter formulation for inverse problems and the continuous formulations. In particular, the linear stability analysis of the moment equations performed in~\cite{HertyVisconti2019} is recalled. In Section~\ref{sec:stabilization} we discuss the stabilization of the dynamics and analyze the properties of the regularized method. In Section~\ref{sec:numerics} we investigate the ability of the method to provide solution to an inverse problem based on a two--dimensional elliptic PDE. Finally, we summarize the results in Section~\ref{sec:conclusion}.

\section{Preliminaries on the Ensemble Kalman Inversion} \label{sec:preliminary}

We briefly recall the original formulation of the Ensemble Kalman Inversion (EKI), cf.~\cite{iglesiaslawstuart2013}, which is based on a sequential update of an ensemble to estimate the solution of control--to--observable inverse problems. The derivation of the method is presented within optimization theory \revision{considering only the deterministic version of EKI, i.e.~without artificial random perturbation of the measurement during the iteration in time.} Focusing on recent continuous limit formulations which have allowed theoretical analysis of the nature of the method, we review the one--dimensional linear stability analysis of the moment equations performed in~\cite{HertyVisconti2019}.

\subsection{Formulation of the Ensemble Kalman {Inversion}} \label{ssec:enkf}

We consider a number  $J$ of ensemble members (realizations of the control $\vec{u}\in\R^d$) combined in $\vec{U}=\left\{\vec{u}^{j} \right\}_{j=1}^J$ with $\vec{u}^j\in\R^d$. The EKI is originally posed as a discrete iteration on $\vec{U}$, derived by solving a  minimization problem that compromises between the background estimate of the given model and additional information provided by data or measurements. For more details, we refer e.g.~to~\cite{iglesiaslawstuart2013}. The iteration index is denoted by $n$ and the collection of the ensemble members by $\vec{U}^n=\{\vec{u}^{j,n}\}_{j=1}^J$, $\forall\,n\geq 0$. The EKI iterates each component of $\vec{U}^n$ at iteration $n+1$ as 
\begin{equation} \label{eq:updateEnKF}
	\begin{aligned}
		\vec{u}^{j,n+1} &= \vec{u}^{j,n} + \vec{C}_{\G}(\vec{U}^n) \left( \vec{D}_{\G}(\vec{U}^n) + \frac1{\Delta t} \vecsym{\Gamma}^{-1} \right)^{-1} (\vec{y} - \G(\vec{u}^{j,n}) ) \\
		%{\vec{y}^{j,n+1}} &= \vec{y} + \vecsym{\xi}^{j,n+1}
	\end{aligned}
\end{equation}
for each $j=1,\dots,J$, where $\Delta t\in\R^+$ is a parameter. In general, each observation or measurement can be perturbed by additive noise~\cite{iglesiaslawstuart2013}. We focus on the case where the measurement data $\vec{y}\in\R^K$ is unperturbed.

The update of the ensemble~\eqref{eq:updateEnKF} requires the knowledge of the operators $\vec{C}_{\G}(\vec{U}^n)$ and $\vec{D}_{\G}(\vec{U}^n)$ which are, in the present finite dimensional setting, the covariance matrices depending on the ensemble set $\vec{U}^n$ at iteration $n$ and on $\G(\vec{U}^n)$, i.e. the image of $\vec{U}^n$ at iteration $n$. More precisely, we have
\begin{equation} \label{eq:covariance}
	\begin{aligned}
		\vec{C}_{\G}(\vec{U}^n) &= \frac{1}{J} \sum_{k=1}^J \left(\vec{u}^{k,n}-\overline{\vec{u}}^n\right) \left(\G(\vec{u}^{k,n})-\overline{\G}^n\right)^\intercal \in \R^{d\times K} \\
		\vec{D}_{\G}(\vec{U}^n) &= \frac{1}{J} \sum_{k=1}^J \left(\G(\vec{u}^{k,n})-\overline{\G}^n\right) \left(\G(\vec{u}^{k,n})-\overline{\G}^n\right)^\intercal \in \R^{K\times K}
	\end{aligned}
\end{equation}
where we define  $\overline{\vec{u}}^n$ and $\overline{\G}^n$ as the mean of $\vec{U}^n$ and $\G(\vec{U}^n)$, respectively:
$$
\overline{\vec{u}}^n = \frac{1}{J} \sum_{j=1}^J \vec{u}^{j,n}, \quad \overline{\G}^n = \frac{1}{J} \sum_{j=1}^J \G(\vec{u}^{j,n}).
$$
The EKI satisfies the subspace property~\cite{iglesiaslawstuart2013}, i.e.~the ensemble iterates stay in the subspace spanned by the initial ensemble. As consequence, the natural estimator for the solution of the inverse problem is provided by the mean of the ensemble.

\subsection{Derivation of the Ensemble Kalman Inversion from an Optimization Point--of--View} \label{ssec:derivation}
The EKI method~\eqref{eq:updateEnKF} for the solution of the inverse problem~\eqref{eq:noisyProb} can be derived by using an optimization point--of--view. % in the following way.
Below, we briefly review the machinery leading to the EKI formulation. For further details and explicit computations we refer to e.g.~\cite{Stuart2019CEnKF}.

We introduce a new variable $\vec{w} = \G(\vec{u}) \in \R^K$ and reformulate~\eqref{eq:noisyProb} equivalently as
\begin{align*}
	\vec{w} &= \G(\vec{u}) \\
	\vec{y} &= \vec{w} + \boldsymbol{\eta}.
\end{align*}
The problem is then reinterpreted as filtering problem by considering a discrete--time dynamical system with state transitions and noisy observations:
\begin{align*}
	\mbox{(dynamics model) } \quad & \
	\begin{cases}
		\vec{u}^{n+1} = \vec{u}^n \\
		\vec{w}^{n+1} = \G(\vec{u}^n)
	\end{cases} \\
	\mbox{(data model) } \quad & \
	\begin{cases}
		\vec{y} = \vec{w}^{n+1} + \boldsymbol{\eta}.
	\end{cases}
\end{align*}
By defining $\vec{v} = [\vec{u},\vec{w}]^T\in\R^{d+K}$ and $\Xi : \vec{v} \mapsto \Xi(\vec{v}) = [\vec{u},\G(\vec{u})]^T \in\R^{d+K}$, the dynamical model can be written as
$$
\vec{v}^{n+1} = \Xi(\vec{v}^n),
$$
whereas the data model becomes
$$
\vec{y} = \vec{H} \vec{v}^{n+1} + \boldsymbol{\eta},
$$
where $\vec{H} = [\vec{0}_{K\times d},\vec{I}_{K\times K}]\in\R^{K\times (d+K)}$ is an observation matrix. Let us denote by $\{ \vec{v}^{j,n} \}_{j=1}^J$ a collection of $J$ ensemble members, also called particles, at time $n$. The method proceeds as follows.

First, the state of all particles at time $n+1$ is predicted using the dynamical model to give $\{ \hat{\vec{v}}^{j,n+1} \}_{j=1}^J$, i.e.~$\hat{\vec{v}}^{j,n+1}=\Xi(\vec{v}^{j,n})$. The resulting empirical covariance $\vec{Cov}\in \R^{(d+K)\times(d+K)}$ of the uncertainties in the predictions is computed. Exploiting the definition of $\vecsym{\Xi}$ it is easy to check that
\begin{equation*}
	\vec{Cov} = \frac{1}{J} \sum_{k=1}^J (\hat{\vec{v}}^{k,n+1} - \overline{\hat{\vec{v}}}^{n+1}) (\hat{\vec{v}}^{k,n+1} - \overline{\hat{\vec{v}}}^{n+1})^\intercal = \begin{bmatrix} \vec{C} & \vec{C}_{\G} \\ \vec{C}^T_{\G} & \vec{D}_{\G}\end{bmatrix},
\end{equation*}
where $\vec{C}_\G$ and $\vec{D}_\G$ are as in~\eqref{eq:covariance}, whereas
\begin{gather*}
	\vec{C} = \frac{1}{J} \sum_{k=1}^J (\vec{u}^{k,n} - \overline{\vec{u}}^{n}) (\vec{u}^{k,n} - \overline{\vec{u}}^{n})^\intercal \in \R^{d\times d}.
\end{gather*}

Then, the update $\vec{v}^{j,n+1}$ of each particle is determined by imposing first order necessary optimality condition of the following minimization problem
$$
\vec{v}^{j,n+1} = \arg\min_{\vec{v}} \mathcal{J}^{j,n}(\vec{v}),
$$
which is solved sequentially, and where $\mathcal{J}^{j,n}(\vec{v})$ is the objective function which encapsulates the model--data compromise:
\begin{equation} \label{eq:objective}
	\mathcal{J}^{j,n}(\vec{v}) = \frac12 \left\| \vec{y}^{n+1} - \vec{H}\vec{v} \right\|^2_{\vecsym{\Gamma}^{-1}} + \frac12 \left\| \vec{v} - \hat{\vec{v}}^{j,n+1} \right\|^2_{\vec{Cov}}.
\end{equation}

Finally, the update~\eqref{eq:updateEnKF} of $\vec{u}^{j,n+1}$, related to the unknown control state only, is obtained as $\vec{H}^\perp \vec{v}^{j,n+1}$, with $\vec{H}^\perp = [\vec{I}_{d\times d},\vec{0}_{d\times K}] \in \R^{d\times(d+K)}$.

This derivation of the ensemble Kalman inversion from  an optimization point--of--view leads into the introduction and motivation of the stabilized formulation of the continuous limit. In particular, we observe that the first term of the objective~\eqref{eq:objective} corresponds to the least--squares functional $\Phi$ given by~\eqref{eq:leastSqFnc}. Therefore, minimization of~\eqref{eq:objective} can be seen as minimization of $\Phi$ subject to a regularization term involving the covariance of the ensemble.

\begin{remark}
	The derivation of the EKI motivated through the optimization approach assumes that the empirical covariance $\vec{Cov}$ is positive definite $\forall\,n\geq0$. In general, it is not possible to guarantee that. In~\cite{Stuart2019CEnKF} and in~\cite{TongMajdaKelly2016,SchillingsPreprint}, this issue is overcome by a constant or time dependent shifting of $\vec{Cov}$.
\end{remark}

\subsection{Continuous Limits of the Ensemble Kalman Inversion} \label{ssec:limitsEnKF}

\subsubsection{Continuous--time}
As in~\cite{schillingsstuart2017}, we compute the continuous--time limit equation of the update~\eqref{eq:updateEnKF}. We consider the parameter $\Delta t$ as an artificial time step for the discrete iteration, i.e.~$\Delta t \sim N_t^{-1}$ with $N_t$ being the maximum number of iterations and define  $\vec{U}^n \approx \vec{U}(n\Delta t)=\left\{\vec{u}^{j}(n\Delta t) \right\}_{j=1}^J$ for $n\geq 0$. Computing the limit $\Delta t\to 0^+$ we obtain
\begin{equation} \label{eq:continuousEnKF1}
	\begin{aligned}
		\frac{\mathrm{d}}{\mathrm{d}t} \vec{u}^j &= \vec{C}_{\G}(\vec{U}) \vecsym{\Gamma} \left( \vec{y} - \G(\vec{u}^j) \right), \quad j=1,\dots,J \\
		\vec{C}_{\G}(\vec{U}) &= \frac{1}{J} \sum_{k=1}^J \left(\vec{u}^{k}-\overline{\vec{u}}\right) \left(\G(\vec{u}^{k})-\overline{\G}\right)^\intercal
	\end{aligned}
\end{equation}
with initial condition $\vec{U}(0) = \vec{U}^0$.
%Exploiting the definition of the operator $\vec{C}_\G(\vec{U})$, system~\eqref{eq:continuousEnKF1} can be restated as
%\begin{equation} \label{eq:continuousEnKF2}
%	\frac{\mathrm{d}}{\mathrm{d}t}  \vec{u}^j = \frac{1}{J} \sum_{k=1}^J \left\langle \G(\vec{u}^k) - \overline{\G} , \vec{y} - \G(\vec{u}^j) \right\rangle_{\vecsym{\Gamma}^{-1}} (\vec{u}^k - \overline{\vec{u}}), \quad j=1,\dots,J
%\end{equation}
%where $\langle \cdot,\cdot \rangle_{\vecsym{\Gamma}^{-1}} = \langle \vecsym{\Gamma}^{\frac12} \cdot,\vecsym{\Gamma}^{\frac12} \cdot \rangle$ and $\langle \cdot,\cdot \rangle$ is the inner product on $\R^K$.
\subsubsection*{\revision{Linear forward model}}
Let us consider the case of $\G$ linear, i.e.~$\G(\vec{u})=\vec{G} \vec{u}$, with $\vec{G}\in\R^{K\times d}$. Then \eqref{eq:continuousEnKF1} is a gradient descent equation and we can write
$
\vec{C}_\G(\vec{U}) = \frac{1}J \sum_{k=1}^J \left(\vec{u}^k-\overline{\vec{u}}\right) \left(\vec{u}^k-\overline{\vec{u}}\right)^T \vec{G}^T.
$
Since the least--squares functional~\eqref{eq:leastSqFnc} yields
\begin{equation} \label{eq:gradientGLinear}
	\nabla_\vec{u} \Phi(\vec{u},\vec{y}) = - \vec{G}^T \vecsym{\Gamma} ( \vec{y} - \vec{G} \vec{u} ),
\end{equation}
equation~\eqref{eq:continuousEnKF1} can be stated in terms of the gradient of $\Phi$ as
\begin{equation} \label{eq:gradientEq}
	\begin{aligned}
		\frac{\mathrm{d}}{\mathrm{d}t} \vec{u}^j &= - \vec{C}(\vec{U}) \nabla_\vec{u} \Phi(\vec{u}^j,\vec{y}), \quad j=1,\dots,J \\
		\vec{C}(\vec{U}) &= \frac{1}J \sum_{k=1}^J (\vec{u}^k-\overline{\vec{u}}) ( \vec{u}^k - \overline{\vec{u}} )^\intercal.
	\end{aligned}
\end{equation}
Equation~\eqref{eq:gradientEq} describes a preconditioned gradient descent equation for each ensemble aiming to minimize $\Phi$. $\vec{C}(\vec{U})$ is positive semi--definite and hence
\begin{equation} \label{eq:decreasePhi}
	\frac{\mathrm{d}}{\mathrm{d}t} \Phi(\vec{u}(t),\vec{y}) = \frac{\mathrm{d}}{\mathrm{d}t} \frac12 \left\| \vecsym{\Gamma}^{\frac12} \left(\vec{y}-\vec{G}\vec{u}\right) \right\|^2 \leq 0.
\end{equation}
Although the forward operator is assumed to be linear, the gradient flow is nonlinear. For further details and properties of the gradient descent equation~\eqref{eq:gradientEq} we refer to~\cite{schillingsstuart2017}. In particular,  the subspace property of the EKI also holds for the continuous dynamics.

Note that, in the continuous--time limit a term originally present in the fully--discrete EKI method~\eqref{eq:updateEnKF} is lost, cf.~\eqref{eq:updateEnKF}.
This is due to the scaling assumption of the measurement covariance by $\Delta t$ which makes the term of order $\Delta t^2$ vanishing in the limit $\Delta t\to 0^+$. This term is however not a Tikhonov regularization--type term but may act as regularization term. We will come back to this point in the next sections.

%In fact, from~\eqref{eq:updateEnKF} we can write
%	\begin{align*}
%	\vec{u}^{j,n+1} =& \vec{u}^{j,n} + \Delta t \vec{C}_{\vec{u}\vec{u}} \vec{G}^T \vecsym{\Gamma} (\vec{y}-\vec{G}\vec{u}^{j,n}) \\ &- \Delta t^2 \vec{C}_{\vec{u}\vec{u}} \vec{G}^T \vecsym{\Gamma} \vec{G} \left( \vec{C}_{\vec{u}\vec{u}}^{-1} + \Delta t \vec{G}^T \vec{\Gamma} \vec{G} \right)^{-1} \vec{G}^T \vecsym{\Gamma} (\vec{y}-\vec{G}\vec{u}^{j,n}) %\\ =& \vec{u}^{j,n} + \Delta t \vec{C}_{\vec{u}\vec{u}} \vec{G}^T \vecsym{\Gamma} \left( \vec{I}_{d\times K} - \Delta t \vec{G} \left( \vec{C}_{\vec{u}\vec{u}}^{-1} + \Delta t \vec{G}^T \vec{\Gamma} \vec{G} \right)^{-1} \vec{G}^T \vecsym{\Gamma} \right) (\vec{y}-\vec{G}\vec{u}^{j,n})
%	\end{align*}

\subsubsection{Mean--field} \label{sssec:MFEnKF}
By definition, the EKI method is a computational method and hence is calculated for a finite ensemble size.  The behavior of the method in the limit of infinitely many ensembles can be studied via mean--field limit leading to a Vlasov--type kinetic equation for the compactly supported on $\R^d$ probability density of $\vec{u}$ at time $t$, denoted by
\begin{equation} \label{eq:kineticf}
	f = f(t,\vec{u}) : \R^+ \times \R^d \to \R^+.
\end{equation}

First we show the limit equation for the case of a non--linear model and later specialize it to a linear model $\vec{G}$. We follow the classical formal derivation to formulate a mean--field equation of a particle system, see~\cite{CarrilloFornasierToscaniVecil2010,hatadmor2008,PareschiToscaniBOOK,Toscani2006}. We introduce the first moments $\vec{m}\in\R^d$, $\vec{m}_\G\in\R^K$ and the second moments $\vec{E}\in\R^{d\times d}$, $\vec{E}_\G\in\R^{d\times K}$ of $f$ at time $t$, respectively, as
\begin{equation} \label{eq:moments}
	\begin{gathered}
		\vec{m}(t) = \int_{\R^d} \vec{u} f(t,\vec{u}) \mathrm{d}\vec{u}, \quad \vec{E}(t) = \int_{\R^d} \vec{u} \otimes \vec{u} f(t,\vec{u}) \mathrm{d}\vec{u}, \\
		\vec{m}_\G(t) = \int_{\R^d} \G(\vec{u}) f(t,\vec{u}) \mathrm{d}\vec{u}, \quad \vec{E}_\G(t) = \int_{\R^d} \vec{u} \otimes \G(\vec{u}) f(t,\vec{u}) \mathrm{d}\vec{u}.
	\end{gathered}
\end{equation}
Since $\vec{u}\in\R^d$, the corresponding discrete measure on the ensemble set $\vec{U} = \left\{ \vec{u}^j \right\}_{j=1}^J$ is given by the empirical measure 
\begin{equation} \label{eq:empiricalf}
	f(t,\vec{u}) = \frac{1}J \sum_{j=1}^J \delta(\vec{u}^j - \vec{u}).
\end{equation}

Let us consider the interacting particle system~\eqref{eq:continuousEnKF1}.
%Defining the operator
%$$
%\vecsym{\C}_\G(\vec{U}) = \frac{1}J \sum_{k=1}^J (\vec{u}^k-\overline{\vec{u}}) \otimes (\G(\vec{u}^k)-\overline{\G})
%$$
%whose corresponding entry is
%$$
%\left( \vecsym{\C}_\G(\vec{U}) \right)_{\kappa,\ell} = \frac{1}J \sum_{k=1}^J u_\kappa^k \G(u^k)_\ell - \overline{u}_\kappa \frac{1}J \sum_{k=1}^J \G(u^k)_\ell - \overline{\G}_\ell \frac{1}J \sum_{k=1}^J u_\kappa^k + \overline{u}_\kappa \overline{\G}_\ell = \frac{1}J \sum_{k=1}^J u_\kappa^k \G(u^k)_\ell - \overline{u}_\kappa \overline{\G}_\ell.
%$$
The empirical measure allows for a mean--field limit of $\vec{C}_\G$ as 
$$
\left(\vec{C}_\G\right)_{\kappa,\ell} = \int_{\R^d} u_\kappa \G(u)_\ell f(t,\vec{u}) \mathrm{d}\vec{u} - \int_{\R^d} u_\kappa f(t,\vec{u}) \mathrm{d}\vec{u} \int_{\R^d} \G(u)_\ell f(t,\vec{u}) \mathrm{d}\vec{u}, \ \kappa,\ell=1,\dots,d
$$
and therefore $\vec{C}_\G$ can be written in terms of the moments~\eqref{eq:moments} of $f$ only as
\begin{equation} \label{eq:covarianceNonLinMeanField}
	\vec{C}_\G(f) = \vec{E}_\G(t) - \vec{m}(t) \vec{m}_\G^\intercal(t) \geq 0.
\end{equation}
We denote a sufficiently smooth test function by $\varphi(\vec{u}) \in C_0^1(\R^d)$ and compute
\begin{align*}
	\frac{\mathrm{d}}{\mathrm{d}t} \left\langle f , \varphi \right\rangle &= \frac{\mathrm{d}}{\mathrm{d}t} \int_{\R^d} \frac{1}{J} \sum_{j=1}^J \delta(\vec{u} - \vec{u}^j) \varphi(\vec{u}) \mathrm{d}\vec{u} = - \frac{1}{J} \sum_{j=1}^J \nabla_\vec{u} \varphi(\vec{u}^j) \cdot \vec{C}_\G(f) \vecsym{\Gamma} (\vec{y}-\G(\vec{u}^j)) \\
	&= - \int_{\R^d} \nabla_\vec{u} \varphi(\vec{u}) \cdot \vec{C}_\G(f) \vecsym{\Gamma} (\vec{y}-\G(\vec{u})) f(t,\vec{u}) \mathrm{d}\vec{u}
\end{align*}
which finally leads to the strong form of the mean-field kinetic equation corresponding to the continuous--time limit~\eqref{eq:continuousEnKF1}:
\begin{equation} \label{eq:kineticFromNonLinEnKF}
	\partial_t f(t,\vec{u}) - \nabla_\vec{u} \cdot \left( \vec{C}_\G(f) \vecsym{\Gamma} (\vec{y}-\G(\vec{u})) f(t,\vec{u}) \right) = 0.
\end{equation}

\subsubsection*{\revision{Linear forward model}}
In case of a linear model $\G(\cdot)=\vec{G}\cdot$\revision{, which is the assumption we use for the subsequent analysis,} the mean--field kinetic equation corresponding to the gradient descent equation~\eqref{eq:gradientEq} becomes
\begin{equation} \label{eq:kineticFromEnKF}
	\partial_t f(t,\vec{u}) - \nabla_\vec{u} \cdot \left( \vec{C}(f) \nabla_\vec{u} \Phi(\vec{u},\vec{y}) f(t,\vec{u}) \right) = 0.
\end{equation}
where, similarly to $\vec{C}_\G(f)$, the operator $\vec{C}(f)$ can be also defined in terms of moments of the empirical measure~\eqref{eq:empiricalf} as
\begin{equation} \label{eq:covarianceMeanField}
	\vec{C}(f) = \vec{E}(t) - \vec{m}(t) \vec{m}^\intercal(t) \geq 0.
\end{equation}
For the rigorous mean--field derivation and analysis of the EKI we refer to~\cite{CarrilloVaes,DingLi2019}.

We observe that~\eqref{eq:kineticFromEnKF} is a nonlinear transport equation arising from non--linear gradient flow interactions and the counterpart of~\eqref{eq:decreasePhi} holds at the kinetic level. Defining
$$
\mathcal{L}(f,\vec{y}) = \int_{\R^d} \Phi(\vec{u},\vec{y}) f(t,\vec{u}) \mathrm{d}\vec{u}
$$
we compute
$$
\frac{\mathrm{d}}{\mathrm{d}t} \mathcal{L}(f,\vec{y}) = \int_{\R^d} \Phi(\vec{u},\vec{y}) \partial_t f(t,\vec{u}) \mathrm{d}\vec{u} = - \int_{\R^d} (\nabla_\vec{u} \Phi(\vec{\vec{u}},\vec{y}) )^T \vec{C}(f) \nabla_\vec{u} \Phi(\vec{\vec{u}},\vec{y}) \mathrm{d}\vec{u} \leq 0
$$
since $\vec{C}(f)$ is positive semi-definite. In particular, $\mathcal{L}(f,\vec{y})$ is decreasing unless $f$ is a Dirac
distribution. This reveals that a solution of $\min_{\vec{u}\in\R^d}\Phi(\vec{u},\vec{y})$ provides a steady solution of the continuous--limit formulation, but the converse is not necessarily true. In fact, all Dirac distributions, i.e.~all $f$ satisfying $\vec{C}(f)=0$, provide steady solutions of~\eqref{eq:kineticFromEnKF}. In particular, the velocity of convergence to the \emph{correct} Dirac distribution may be highly influenced by the initial properties of the initial condition, i.e.~by the distribution of the initial ensemble.
%These considerations are in--depth investigated at the level of the moments of the kinetic equation in the following section.}

\subsection{Stability of the Moment Equations} \label{ssec:stabilityEnKF}

The previous observations %concerning the existence of infinitely many Dirac--delta--type steady states can be
are in--depth investigated %also
at the level of moments of the mean--field equation, in order to gain insights on the nature of the steady states. We recall that the expected value of the ensemble is selected as estimator for the solution, due to the subspace property satisfied by the EKI. For this reason the analysis of moments is of crucial importance in order to understand the properties of the method.

In~\cite{HertyVisconti2019}, the linear stability analysis of the moment equations resulting from~\eqref{eq:kineticFromEnKF} has been investigated for one--dimensional controls. Here, we briefly review, but also extend, that analysis. First, we restrict the attention to the case $d=K=1$. From now on, we avoid using bold font to emphasize \revision{that} the involved quantities are one--dimensional.

The dynamical system for the first and second moment is computed from~\eqref{eq:kineticFromEnKF}. Using the linearity of the model we obtain
\begin{equation} \label{eq:momentODEs}
	\begin{aligned}
		\frac{\mathrm{d}}{\mathrm{d}t} m(t) &=  C(m,E) G^T \Gamma ( y - G m ) \\
		\frac{\mathrm{d}}{\mathrm{d}t} E(t) &= 2 C(m,E) G^T \Gamma ( y m - G E ), 
	\end{aligned}
\end{equation}
where we observe that $C(m,E)$ corresponds to the variance, in fact $$\int_{\R} (u-m(t))^2 f(t,u) \mathrm{d}u = E(t) - m(t)^2.$$ System~\eqref{eq:momentODEs} is closed by the second moment equation \revision{due to the assumption of a linear model. If the forward model were nonlinear, then the knowledge of the dynamics of the mixed moments in~\eqref{eq:moments} would be required to have a closed hierarchy.}

We analyze steady--states and their stability with $G=\Gamma=1$. Nullclines of~\eqref{eq:momentODEs} are \revision{given by}
\begin{gather*}
	\revision{\frac{\mathrm{d}}{\mathrm{d}t} m(t) = 0 \ \Leftrightarrow \ m = y \, \lor \, E = m^2} \\
	\revision{\frac{\mathrm{d}}{\mathrm{d}t} E(t) = 0 \ \Leftrightarrow \ E = ym \, \lor \, E = m^2}.
\end{gather*}
Equilibrium points arise by intersection of the nullclines and are \revision{thus}
\begin{equation} \label{eq:equilibriumEnKF}
	F_k = (k,k^2), \ k\in\R,
\end{equation}
i.e.~all equilibria are points on the set $E=m^2$ for which $C = 0$. We note that they lie on the boundary of the admissible region $C \geq 0$. This means that all Dirac delta distributions are steady--states of the mean--field equation, as shown for cases of arbitrary dimension at the end of Section~\ref{sssec:MFEnKF}. As consequence we have a set of infinitely many steady--states. The one minimizing the least square functional $\Phi$ is $\delta(u-y)$, corresponding to $F_y=(y,y^2)$. Studying the linear \revision{(in)}stability of the equilibrium points, it is simple to show that all the $F_k$'s have double--zero
eigenvalues and, thus are non--hyperbolic Bogdanov--Takens--type equilibria.

Nevertheless, we point out that the desired steady state $F_y$, although being a Bogdanov--Takens equilibrium, is still a global attractor of the dynamics when the initial condition belongs to the region $C>0$.
In order to show this, consider the coupled system of ODEs for $m$ and $C=E-m^2$, with $G=\Gamma=1$:
\begin{align*}
	\frac{\mathrm{d}}{\mathrm{d}t} m(t) &=  C ( y - m ) \\
	\frac{\mathrm{d}}{\mathrm{d}t} C(t) &= -2 C^2.
\end{align*}
Then, we observe that, on the region $C>0$, $\dot{m}>0$ for $m<y$ and $\dot{m}<0$ for $m>y$. Instead, $\dot{C}<0$ for all $C$. The trajectories in the phase space $(m,C)$ can be computed by solving
$$
\frac{\mathrm{d}}{\mathrm{d}m} C(m) = - 2 \frac{C}{(y-m)}
$$
which gives $C(m) = c(m-y)^2$, with $c$ an integration constant. This proves that the equilibrium $(m,C)=(y,0)$, which corresponds to $F_y$ in the phase space $(m,E)$, is the global attractor for all initial conditions $C>0$.

However, having non--hyperbolic equilibria has several undesirable consequences: 
Bogdanov--Takens equilibria are not asymptotically stable, in fact their linearization is unstable. More importantly, they 
are non--hyperbolic and thus structurally unstable, i.e.~susceptible to qualitative changes under arbitrary small perturbations of the underlying model. In addition, the instability of the phase space may result also in numerical instabilities leading trajectories to the unfeasible region, i.e.~where $C(m,E)<0$, or to get stuck in an equilibrium point which is not $F_y$. Finally, we may face a slow convergence to the global attractor equilibrium. In fact, we notice that the differential equation for $C$ can also be solved explicitly, giving us
\begin{align*}
	C(t) = \frac{C(0)}{1+2 C(0)t},
\end{align*}
with $C(0)$ being the initial condition. Thus, $C(t)\to 0^+$, as $t\to+\infty$, very slowly, precisely with rate $\mathcal{O}(t^{-1})$.

\subsubsection{Stability for Multi--Dimensional Controls} \label{sssec:multidStabilityEKI}
We recall that the existence of infinitely many steady--states satisfying $\vec{C}=\vec{0}$ holds in general dimension, as observed in Section~\ref{sssec:MFEnKF}. Instead, the above stability analysis is performed in the simplest setting $d=1$. \rrevision{However, we can show that the nature of the moment equilibria} lying in the kernel of $\vec{C}$ is maintained also in general dimension, i.e.~when $d>1$.

Let $\vec{y}\in\R^d$ and consider the $d+d^2$ dynamical system
\begin{equation} \label{eq:multidMoments}
	\rrevision{
		\begin{aligned}
			\frac{\mathrm{d}}{\mathrm{d}t} \vec{m}(t) &= (\vec{E}-\vec{m}\vec{m}^T)(\vec{y}-\vec{m}) \\
			\frac{\mathrm{d}}{\mathrm{d}t} \vec{E}(t) &= (\vec{E}-\vec{m}\vec{m}^T)(\vec{y}\vec{m}^T-\vec{E}) + (\vec{m}\vec{y}^T-\vec{E})(\vec{E}-\vec{m}\vec{m}^T),
		\end{aligned}
	}
\end{equation}
where $\vec{m}\in\R^d$ and $\vec{E}\in\R^{d\times d}$. We are interested in the equilibrium points $(\vec{m}^*,\vec{E}^*)$ such that $\vec{C}^*:=\vec{E}^*-\vec{m}^*(\vec{m}^*)^T = \vec{0}$. To this end, it is convenient to rewrite, equivalently, the dynamical system~\eqref{eq:multidMoments} in terms of $\vec{m}$ and $\vec{C}=\vec{E}-\vec{m}\vec{m}^T\in\R^{d\times d}$ as
\begin{equation} \label{eq:multidMoments2}
	\rrevision{
		\begin{aligned}
			\frac{\mathrm{d}}{\mathrm{d}t}\vec{m}(t) =& \vec{C} (\vec{y} - \vec{m}) \\
			\frac{\mathrm{d}}{\mathrm{d}t}\vec{C}(t) =& \frac{\mathrm{d}}{\mathrm{d}t}\vec{E} - \frac{\mathrm{d}}{\mathrm{d}t}\vec{m} \vec{m}^T - \vec{m} \frac{\mathrm{d}}{\mathrm{d}t}\vec{m}^T\\
			=& -2 \vec{C} \vec{C}.
		\end{aligned}
	}
\end{equation}
It is clear that $\vec{C}=\vec{0}$ defines a set of equilibrium points of~\eqref{eq:multidMoments2}. By linearization we compute the Jacobian matrix
$$
\mathbf{J} = \begin{bmatrix}
\mathbf{J}_{11} & \mathbf{J}_{12} \\
\mathbf{J}_{21} & \mathbf{J}_{22}
\end{bmatrix} \in \mathbb{R}^{(d+d^2)\times (d+d^2)}
$$
where the four blocks are
\begin{gather*}
	\mathbf{J}_{11} = \frac{\partial \dot{\vec{m}}}{\partial \vec{m}} \in \mathbb{R}^{d\times d}, \quad \mathbf{J}_{12} = \frac{\partial \dot{\vec{m}}}{\partial \vec{C}} \in \mathbb{R}^{d\times d^2}, \\
	\mathbf{J}_{21} = \frac{\partial \dot{\vec{C}}}{\partial \vec{m}} \in \mathbb{R}^{d^2\times d},\quad \mathbf{J}_{22} = \frac{\partial \dot{\vec{C}}}{\partial \vec{C}} \in \mathbb{R}^{d^2\times d^2}.
\end{gather*}
\rrevision{Then, on $\vec{C}=\vec{0}$ we have $\mathbf{J}_{11} = -\vec{C} = \vec{0}$ and, similarly,  $\mathbf{J}_{22} = \vec{0}$ since it is linear with respect to $\vec{C}$. Instead,
	$$
	\mathbf{J}_{12} = \begin{bmatrix} \mathbf{M}_1, \dots, \mathbf{M}_d\end{bmatrix}, \quad \mathbf{M}_i = \vec{e}_i (\vec{y}-\vec{m})^T \in \mathbb{R}^{d\times d}
	$$
	where $\vec{e}_i\in\mathbb{R}^d$ is the $i$--th vector of the standard basis of the Euclidean space $\mathbb{R}^d$. Then, in general, $\mathbf{J}_{12}\neq 0$ also on $\vec{C}=\vec{0}$, but $\mathbf{J}_{21} \equiv \vec{0}$. We conclude that the Jacobian $\vec{J}$ has $d$ zero eigenvalues on $\vec{C}=\vec{0}$, corresponding to non--hyperbolic steady states of Bogdanov--Takens type. Therefore, all the undesirable consequences discussed in the one--dimensional setting occur also in general dimension.}

\section{Stabilization of the Dynamics} \label{sec:stabilization}

Our goal is to introduce a modified formulation of the continuous dynamics in order to make the \rrevision{target} equilibrium a globally asymptotically stable equilibrium of the system of moment equations. We are only interested in the \rrevision{target} equilibrium since the others are irrelevant for the optimization. The modification we propose  is inspired by the idea to restore the regularization effect of the discrete EKI which gets lost in the continuous limits.

\rrevision{Given $\vecsym{\Sigma}\in\R^{d\times d}$ symmetric,} we propose to consider the following general discrete dynamics for each ensemble member $j=1,\dots,J$:
\begin{equation} \label{eq:stableEnKF}
	\begin{aligned}
		\frac{\mathrm{d}}{\mathrm{d}t} \vec{u}^j &= \tilde{\vec{C}}_\G(\vec{U}) \vecsym{\Gamma} (\vec{y}-\G(\vec{u}^j)) + R(\vec{U}), \\
		R(\vec{U}) &= \beta \tilde{\vec{C}}(\vec{U})(\vec{u}^j-\bar{\vec{u}}), \\
		\tilde{\vec{C}}_\G(\vec{U}) &= \rrevision{\vec{C}_\G(\vec{U}) + (1-\alpha) \vecsym{\Sigma}_\G}, \\
		\tilde{\vec{C}}(\vec{U}) &= \rrevision{\vec{C}(\vec{U}) + (1-\alpha) \vecsym{\Sigma}},
	\end{aligned}
\end{equation}
with \rrevision{$\alpha,\beta\in\R$ and where $\vecsym{\Sigma}_G:=\left(\G(\vecsym{\Sigma})\right)^T\in\R^{d\times K}$ with $\G$ acting on the columns of $\vecsym{\Sigma}$. The choices $\alpha=1$ and $\beta=0$ yield the} continuous--time limit~\eqref{eq:continuousEnKF1} for the original ensemble Kalman inversion. Since the analysis of stability of the new dynamics~\eqref{eq:stableEnKF} will be performed again in the case of a linear model, now we specialize~\eqref{eq:stableEnKF} to this particular setting and then we discuss the role of the modifications we propose.
\subsubsection*{\rrevision{Linear forward model and mean--field limit}} \revision{The analysis of the proposed stabilization will be also performed in the linear setting} $\G(\cdot)=\vec{G}\cdot$. \revision{Then,} \eqref{eq:stableEnKF} becomes
\begin{equation} \label{eq:stableLinearEnKF}
	\begin{aligned}
		\frac{\mathrm{d}}{\mathrm{d}t} \vec{u}^j &= - \tilde{\vec{C}}(\vec{U}) \nabla_\vec{u} \Phi(\vec{u}^j,\vec{y}) + R(\vec{U}), \\
		R(\vec{U}) &= \beta \tilde{\vec{C}}(\vec{U})(\vec{u}^j-\bar{\vec{u}}), \\
		\tilde{\vec{C}}(\vec{U}) &= \rrevision{\vec{C}(\vec{U}) + (1-\alpha)\vecsym{\Sigma}},
	\end{aligned}
\end{equation}
with $\Phi$ being the least--squares functional~\eqref{eq:leastSqFnc}.

The modified dynamics~\eqref{eq:stableLinearEnKF} differs from the standard continuous--time limit of the discrete EKI, cf.~\eqref{eq:gradientEq}, in the formulation of the preconditioner $\tilde{\vec{C}}(\vec{U})$ and in the presence of the additive term $R(\vec{U})$. The new preconditioner can be thought, for \rrevision{$\alpha<1$}, as inflation of the covariance $\vec{C}(\vec{U})$ defined in~\eqref{eq:gradientEq}. This modification allows us to stabilize the phase space of the moments and \rrevision{$\alpha$} plays the role of a regularization/bifurcation parameter. The term $R(\vec{U})$, instead, can be thought as acceleration/relaxation to equilibrium. \revision{This formal presentation of the role of the parameters will be made mathematically rigorous in the following, being the core of the analysis in Section~\ref{ssec:stabilitySEnKF}, Section~\ref{ssec:varianceInflation} and Section~\ref{ssec:analysisEnsemble}.}

%We point out that we do not claim uniqueness of the modeling choice of the stabilization term. In fact, we will see that other forms of covariance inflation allow stabilization of the phase space, but they may lead to unbounded choice of the regularization parameter. See Section~\ref{ssec:varianceInflation}.

%\begin{remark} \label{rm:others}
%	We do not claim uniqueness for the modeling choice of the term $R(\vec{U})$. For instance $R(\vec{U})=\beta \vec{u}^j$ would correspond to a Tikhonov--type regularization studied in~\cite{ChadaStuartTong2019}. Alternatively,  \cite{HertyVisconti2019}  modified the discrete dynamics with additive white Gaussian noise. This approach leads to a Fokker--Planck--type equation, where Dirac delta distributions are no longer steady states and the desired equilibrium, $F_y$, depends on the nonzero variance $\sigma$ of the noise and becomes $(y,y^2\pm \sqrt{2\sigma^2})$. 
%\end{remark}

\revision{Performing formal computations as in Section~\ref{sssec:MFEnKF}, it is possible to show that the solution of the dynamical system~\eqref{eq:stableLinearEnKF} satisfies the weak form of the following mean--field equation:
	\begin{equation} \label{eq:stableMeanField}
		\begin{aligned}
			\partial_t f(t,\vec{u}) - \nabla_\vec{u} \cdot \left( \tilde{\vec{C}}(f) \left( \nabla_\vec{u} \Phi(\vec{u},\vec{y}) - \beta (\vec{u}-\vec{m}) \right) f(t,\vec{u}) \right)= 0,
		\end{aligned}
	\end{equation}
	where $\tilde{\vec{C}}(f)$ is the mean--field interpretation of $\tilde{\vec{C}}(\vec{U})$ in~\eqref{eq:stableLinearEnKF}. In fact, via the empirical measure~\eqref{eq:empiricalf} we have
	\begin{align*}
		\left( \tilde{\vec{C}}(\vec{U}) \right)_{i,\ell} =& \int_{\R^d} u_i u_\ell f(t,\vec{u}) \mathrm{d}\vec{u} - \int_{\R^d} u_i f(t,\vec{u}) \mathrm{d}\vec{u} \int_{\R^d} u_\ell f(t,\vec{u}) \mathrm{d}\vec{u} \\ & \rrevision{+(1-\alpha)\vecsym{\Sigma}}, \quad i,\ell=1,\dots,d
	\end{align*}
	and therefore it can be written in terms of the moments of $f$ only, leading to}
$$
\tilde{\vec{C}}(f) = \vec{E}(t) - \vec{m}(t)\vec{m}^\intercal(t) \rrevision{+(1-\alpha)\vec{\Sigma}}.
$$
Note that, if $\alpha=1$, $\tilde{\vec{C}}(f) \equiv \vec{E}(t) - \vec{m}(t)\vec{m}^\intercal(t) = \vec{C}(f) \geq 0$, see~\eqref{eq:covarianceMeanField}, and that $\tilde{\vec{C}}$ can be seen as inflation of the covariance $\vec{C}$ for $\alpha<1$.

\subsection{\rrevision{One--dimensional Stability Analysis of the Moment Equations}} \label{ssec:stabilitySEnKF}

The stability of moments is again studied in the simplest setting of a one--dimensional problem, i.e.~$d=K=1$. From~\eqref{eq:stableMeanField} we compute
\begin{equation} \label{eq:stableMoments}
	\begin{aligned}
		\frac{\mathrm{d}}{\mathrm{d}t} m &= \tilde{C}(m,E) G^T \Gamma ( y - Gm ) \\
		\frac{\mathrm{d}}{\mathrm{d}t} E &= 2 \tilde{C}(m,E) \left( G^T \Gamma ( y m  - GE ) 
		+ \beta C(m,E) \right)\\
		\tilde{C}(m,E) &= \rrevision{E - m^2 + (1-\alpha)\sigma}\\
		C(m,E) &= E - m^2,
	\end{aligned}
\end{equation}
where we again avoid the use of bold fonts to highlight the one--dimensional quantities. \rrevision{Here $\sigma$ is a positive scalar and for instance one could take $\sigma=m^2$ so that $\tilde{C}(m,E) = E - \alpha m^2$. Without loss of generality we consider this choice in the subsequent analysis.}

We introduce the following concepts of admissible initial conditions, solutions and equilibria of~\eqref{eq:stableMoments}.
\begin{mydef} \label{def:admissible}
	We say that an initial condition of~\eqref{eq:stableMoments} is admissible if $(m(0),E(0))\in\R\times\R^+$ and $E(0)> m(0)^2$.\\
	\indent
	We say that a solution of~\eqref{eq:stableMoments} with admissible initial condition is admissible or belongs to the feasible domain of the phase space $(m,E)$ if $t\in\R^+\mapsto(m(t),E(t))\in\R\times\R^+$ satisfies $E(t)\geq m(t)^2$, $\forall\,t>0$.\\
	\indent
	We say that $F=(m_\infty,E_\infty)$ is an admissible or feasible equilibrium point of~\eqref{eq:stableMoments} if $(\dot{m},\dot{E})|_{F}\equiv 0$ and $F\in\R\times\R^+$ with $E_\infty \geq m_\infty^2$.
\end{mydef}
Definition~\ref{def:admissible} can be generalized to the dynamics in general dimension.

Now, we analyze the behavior of the phase portrait of the dynamical system~\eqref{eq:stableMoments}.
\begin{proposition} \label{th:stability}
	Let $(m(0),E(0))=(m_0,E_0)$ be an admissible initial condition of~\eqref{eq:stableMoments}. Assume that $G=\Gamma=1$. Then, the dynamical system has two feasible equilibrium points, $F_y=(y,y^2)$ and $F_{0,\alpha}=(0,0)$, $\forall\,\alpha\in\R$. In particular, $F_{0,\alpha}$ is a non--hyperbolic Bogdanov--Takens equilibrium, and $F_y$ is an asymptotically stable equilibrium, namely $\exists\,\delta>0$ such that if $\| (m_0,E_0) - F_y \| < \delta$ then $\lim_{t\to\infty} (m(t),E(t))=F_y$, provided $\alpha<1$ and $\beta<1$.	
\end{proposition}
\begin{proof}
	Steady states are obtained as intersection of the nullclines. For the system~\eqref{eq:stableMoments} with $G=\Gamma=1$ we get the following equilibrium points $(m,E)$:
	\begin{equation} \label{eq:equilibriumStable}
		F_y=(y,y^2), \quad F_{k,\alpha}=(k,\alpha k^2), \ k\in\R.
	\end{equation}
	By the Hartman--Grobman Theorem, non--linear dynamical systems are locally topologically conjugate to their linearized formulations near hyperbolic fixed point. Thus, if $F_y$ is hyperbolic and asymptotically stable, the local phase portrait of the non--linear system~\eqref{eq:stableMoments} is equivalent to that of its linearization $[\dot{m},\dot{E}]^T=\vec{J}(m,E)[m,E]^T$, where $\vec{J}(m,E)$ is the Jacobian. 
	%Thus, we analyze stability of~\eqref{eq:equilibriumStable} by looking at the sign of the eigenvalues of the Jacobian of~\eqref{eq:stableLinearEnKF} with $G=\Gamma=1$:
	%\begin{align*}
	%	\vec{J}(m,E) &= \begin{bmatrix}
	%	3\alpha m^2 - 2\alpha y m - E & y-m \\
	%	&\\
	%	\vec{J}_{21}(m,E) & \vec{J}_{22}(m,E)
	%	\end{bmatrix} \\[0.25cm]
	%	\vec{J}_{21}(m,E) &= (y - 2\beta m) (- 2\alpha m^2 + 2E) - 4\alpha m (m y - E + \beta(- m^2 + E))\\[0.25cm]
	%	\vec{J}_{22}(m,E) &= 2 m y - 2 E + 2 \beta (- m^2 + E) + (- 2\alpha m^2 + 2E)(k2 - 1). 
	%\end{align*}
	The eigenvalues of the Jacobian evaluated at the equilibrium point $F_y$ are
	\begin{align}
		\lambda_1^{F_y}=y^2(\alpha-1), \quad \lambda_2^{F_y}=-2y^2(\beta-1)(\alpha-1),
	\end{align}
	therefore $F_y$ is an asymptotically stable equilibrium if $\alpha<1$ and $\beta<1$. We notice that $\alpha<1$ automatically implies that, except for $k=0$,  all the other equilibria $F_{k,\alpha}$ are not feasible since then $C(m,E)<0$.  In addition, $F_{0,\alpha}$ is still a Bogdanov--Takens--type equilibrium, and thus unstable, since its eigenvalues are	$\lambda_{1,2}^{F_{0,\alpha}}=0$.
\end{proof}

The previous result shows that $\alpha$ plays the role of a bifurcation parameter. In fact, for $\alpha \to 1$ we recover the same equilibria and topological behavior in the phase space as in the classical continuous--limits of the ensemble Kalman inversion, cf.~\eqref{eq:equilibriumEnKF}. In particular, $\alpha$ allows to stabilize the dynamics. The new ensemble update~\eqref{eq:stableLinearEnKF}
still has infinitely many equilibria but on the set $E=\alpha m^2$, which lies in the unfeasible region of the phase space for $\alpha<1$. In addition, for this choice of $\alpha$, \eqref{eq:equilibriumEnKF} has $F_y$ as isolated hyperbolic and asymptotically stable fixed point of the dynamics.

\begin{remark}[Lower bound for $\alpha$]	
	We observe that we have not discussed the stability of the equilibrium points $F_{k,\alpha}$, since the choice $\alpha<1$ makes them unfeasible. If we impose that they are unstable, then we need additional constraints on $\alpha$ and $\beta$. The eigenvalues of a point $F_{k,\alpha}$, $k\in\R\setminus\{0\}$, are
	$$
	\lambda_1^{F_{k,\alpha}} = 0, \quad \lambda_2^{F_{k,\alpha}} = 2 k (y-k\beta)(1-\alpha).
	$$
	Thus, these points are still non--hyperbolic, saddle--node--type equilibria and unstable from all approaching trajectories if $\beta<\frac{y}{k}$, when $k>0$, and if $\beta>\frac{y}{k}$, when $k<0$. In particular, if $k=\frac{y}{\alpha+\beta-\alpha\beta}$, we require that $\alpha\geq 0$, obtaining a lower bound for the bifurcation parameter $\alpha$.
\end{remark}

Finally, we are ready to prove the global asymptotic stability of $F_y$, which is guaranteed by the following proposition.
\begin{proposition} \label{th:globalStability}
	The point $F_y=(y,y^2)$ is a globally asymptotically stable equilibrium of the dynamical system~\eqref{eq:stableMoments} with $G=\Gamma=1$, namely $\lim_{t\to\infty} (m(t),E(t))=F_y$ for any admissible initial condition $(m(0),E(0))$, provided $\alpha<1$ and $\beta<1$.
\end{proposition}
\begin{proof}
	To prove global asymptotic stability for $F_y$ we note that the relevant subset of the phase space in $\R^2$ is bounded by $E = m^2$. It is easy to see that the vector field generated by \eqref{eq:stableMoments} for $\alpha < 1$ is always pointing inwards, i.e.~using a Lyapunov stability argument, there exists a suitable function $V$ defined on $E > m^2$ such that $\frac{\mathrm{d}}{\mathrm{d}t} V(m,E)<0$. For instance, take $V(m,E)=(m-y)^2+(E-y^2)^2$. In addition, for large enough $E$ and $\beta < 1$ solutions do not escape to infinity. Since the only equilibria in this region are on the boundary and since $F_y$ is the only locally stable equilibrium, hence by the Poincar\'{e}--Bendixson Theorem we have the statement.
\end{proof}

\begin{remark}
	\rrevision{As already observed, we do not claim uniqueness of the choice of the covariance inflation that is responsible for the stabilization of the dynamics.}
	Alternatively, \rrevision{another form of stabilization was proposed in~\cite{HertyVisconti2019} where the authors} modified the discrete dynamics with additive white Gaussian noise. This approach leads to a Fokker--Planck--type equation, where Dirac delta distributions are no longer steady states and the desired equilibrium, $F_y$, depends on the nonzero variance $\eta$ of the noise and becomes $(y,y^2\pm \sqrt{2\eta^2})$. Therefore, this form of stabilization does not preserve the desired equilibrium.
\end{remark}

\subsubsection{Decay Rate}
Proposition~\ref{th:stability} and Proposition~\ref{th:globalStability} require $\beta<1$. Therefore, $R(\vec{U})$, introduced in the discrete dynamics~\eqref{eq:stableEnKF}, is not needed to stabilize the phase portrait since $\beta=0$ is admissible. Although an optimal, i.e.~exponential, rate of convergence is already guaranteed by the fact that the desired equilibrium $F_y$ is globally asymptotically stable, we show in this section that the term $R(\vec{U})$ allows to further speed up the convergence to $F_y$. In particular, we observe an improvement with respect to the rate of convergence obtained by the classical continuous--limit of the EKI, which is $\mathcal{O}(t^{-1})$ as observed at the end of Section~\ref{ssec:stabilityEnKF}.

We notice that $F_y$ is also a point of the phase space where $C\equiv 0$, as it happens for the classical EKI formulation. Therefore, in order to study the convergence rate to $F_y$, we study the decay speed of the variance $C(t)=E-m^2$ at $0^+$.
%between the continuous--time limit of the classical EKI formulation and the stabilized dynamics.
%From the classical {EKI} dynamics~\eqref{eq:momentODEs} we compute \todo{Remeber this is now in section 2}
%\begin{align*}
%\frac{\mathrm{d}}{\mathrm{d}t} C = \frac{\mathrm{d}}{\mathrm{d}t} E - 2 m \frac{\mathrm{d}}{\mathrm{d}t} m = -2C^2
%\end{align*}
%and thus $C$ is decreasing in time with rate $\mathcal{O}(t^{-1})$. In fact, the solution \revision{can be explicitly given as}
%\begin{align*}
%C(t) = \frac{C(0)}{1+2 C(0)t},
%\end{align*}
%\revision{with $C(0)$ being the initial condition.}
%Instead, in
For the stabilized version of the ensemble dynamics, cf.~\eqref{eq:stableMoments}, we compute
\begin{align*}
	\frac{\mathrm{d}}{\mathrm{d}t} C = \frac{\mathrm{d}}{\mathrm{d}t} E - 2 m \frac{\mathrm{d}}{\mathrm{d}t} m &= -2(1-\beta)\tilde{C} C \leq -2(1-\beta)(1-\alpha)m^2C
\end{align*}
provided $\beta<1$. Applying Gronwall inequality we obtain
\begin{align*}
	C(t) \leq C(0)\exp\left( -2(1-\beta)(1-\alpha)\int_0^t m^2(s)\mathrm{d}s \right)
\end{align*}
which implies exponential rate of decay to $0^+$ for $t\to\infty$. In particular, we observe that the exponential decay can be obtained without the acceleration term $R(\vec{U})$ as well, i.e.~taking $\beta=0$, but the decay is faster for $\beta<0$.

We conclude that, while $\alpha$ plays the role of a bifurcation and stabilization parameter leading to a change of the equilibria in the phase space, $\beta$ plays the role of a relaxation or acceleration parameter speeding--up the convergence to the desired equilibrium $F_y$. Therefore, the stabilized EKI proposed in this work has two advantages: it is robust and the approach to the right equilibrium is exponentially fast.

\subsection{\rrevision{Multi--Dimensional Stabilization}} \label{ssec:varianceInflation}
\rrevision{For the choice $\sigma=m^2$ made in the previous section, the corresponding multi--dimensional inflation of the covariance matrix would be written as $\tilde{\vec{C}} = \vec{C} + (1-\alpha) \vecsym{\Sigma}$ with $\vecsym{\Sigma} = \vec{m}\vec{m}^T$, and thus $\tilde{\vec{C}} = \vec{E} - \alpha \vec{m}\vec{m}^T$. In this case, $\vec{\Sigma}$ is a rank--one matrix. We will show that in the general dimension setting $\vecsym{\Sigma}$ is required to be a full--rank matrix in order to have an unfolding which makes the ensemble dynamics hyperbolic.}

\rrevision{To this end, let us first consider $\vecsym{\Sigma} = \vec{m}\vec{m}^T$. With the same assumptions of Section~\ref{sssec:multidStabilityEKI} we compute the $d+d^2$ dynamical system of the moments of~\eqref{eq:stableLinearEnKF} taking $\beta=0$ :
	\begin{align*}
		\frac{\mathrm{d}}{\mathrm{d}t} \vec{m}(t) &= (\vec{E} - \alpha \vec{m} \vec{m}^T) (\vec{y}-\vec{m}) \\
		\frac{\mathrm{d}}{\mathrm{d}t} \vec{E}(t) &= (\vec{E}- \alpha \vec{m} \vec{m}^T) (\vec{y} \vec{m}^T - \vec{E}) + (\vec{m} \vec{y}^T - \vec{E}) (\vec{E}- \alpha \vec{m} \vec{m}^T).
	\end{align*}
	In terms of $(\vec{m},\vec{C})$, where we recall that $\vec{C}=\vec{E}-\vec{m} \vec{m}^T$, we have
	\begin{equation} \label{eq:momentsQuasiStable}
		\begin{aligned}
			\frac{\mathrm{d}}{\mathrm{d}t} \vec{m}(t) &= \vec{C} (\vec{y}-\vec{m}) + (1-\alpha) \vec{m} \vec{m}^T (\vec{y}-\vec{m})\\
			\frac{\mathrm{d}}{\mathrm{d}t} \vec{C}(t) &= -2\vec{C} \vec{C} - (1-\alpha) \vec{m} \vec{m}^T \vec{C} - (1-\alpha) \vec{C} \vec{m} \vec{m}^T.
		\end{aligned}
	\end{equation}
}

\rrevision{We focus on the target equilibrium $(\vec{m}^*,\vec{C}^*)=(\vec{y},\vec{0})$ which is still a critical point of the moment dynamics~\eqref{eq:momentsQuasiStable}. The problem of the zero eigenvalue which makes non--hyperbolic the target equilibrium can be illustrated already using just the equation for $\dot{\vec{m}}$. For $\vec{m} = \vec{m}^* + \delta\vec{m}$ and $\vec{C} = \vec{C}^* + \delta\vec{C}$, the perturbation to first order becomes
	$$
	\frac{\mathrm{d}}{\mathrm{d}t} \delta\vec{m}(t) = -(1-\alpha) \vec{y} \vec{y}^T \delta \vec{m}.
	$$
	Then a zero eigenvalue occurs if the matrix $(1-\alpha) \vec{y} \vec{y}^T$ is singular and hence any vector for $\delta \vec{m}$ that is in the kernel of this matrix has a zero eigenvalue. Thus, for any choice of $\alpha$ there are $d-1$ zero eigenvalues since $\vec{y} \vec{y}^T$ has rank 1.}

\rrevision{Similarly, for the $\dot{\vec{C}}$ equation, we get
	$$
	\frac{\mathrm{d}}{\mathrm{d}t} \delta\vec{C}(t) = (1- \alpha) \vec{y} \vec{y}^T \delta \vec{C}  - (1-\alpha) \delta \vec{C} \vec{y} \vec{y}^T,
	$$
	and again, since $\vec{y} \vec{y}^T$ is a rank 1 matrix, we have $d-1$ zero eigenvalues. 
	Hence, for the two moment equations in $d$ dimensions, i.e.~for $d^2 + d$ equations, we have $2(d-1)$ zero eigenvalues corresponding to the target equilibrium $(\vec{m}^*,\vec{C}^*)$.}

\rrevision{We conclude that an inflation such as $\vecsym{\Sigma} = \vec{m}\vec{m}^T$ can not make the target equilibrium hyperbolic. The next goal is to show that the problem of this choice is that $\Sigma$ is a rank--one matrix. In fact, for a general $\vecsym{\Sigma}$, the moment dynamics~\eqref{eq:momentsQuasiStable} writes
	\begin{equation*}
		\begin{aligned}
			\frac{\mathrm{d}}{\mathrm{d}t} \vec{m}(t) &= \vec{C} (\vec{y}-\vec{m}) + (1-\alpha) \vecsym{\Sigma} (\vec{y}-\vec{m})\\
			\frac{\mathrm{d}}{\mathrm{d}t} \vec{C}(t) &= -2\vec{C} \vec{C} - (1-\alpha) \vecsym{\Sigma} \vec{C} - (1-\alpha) \vec{C} \vecsym{\Sigma},
		\end{aligned}
	\end{equation*}
	and the linearization at the target equilibrium $(\vec{m}^*,\vec{C}^*)$ is
	\begin{align*}
		\frac{\mathrm{d}}{\mathrm{d}t} \delta\vec{m}(t) &=  - (1-\alpha) \vecsym{\Sigma}  \delta \vec{m} \\
		\frac{\mathrm{d}}{\mathrm{d}t} \delta\vec{C}(t) &= - 4 (1-\alpha) \vecsym{\Sigma} \delta \vec{C}.
	\end{align*}
	Hence for $\alpha<1$ and $\vecsym{\Sigma}$ full--rank the target equilibrium becomes hyperbolic.}

\subsection{Properties of the Ensemble Dynamics} \label{ssec:analysisEnsemble}
It is possible to provide a gradient flow interpretation also for the stabilized dynamics~\eqref{eq:stableLinearEnKF}. In fact, we observe that each ensemble is solving a preconditioned gradient descent equation of the type
\begin{equation*}
	\begin{aligned}
		\frac{\mathrm{d}}{\mathrm{d}t} \vec{u}^j &= -{\bf \tilde{C}}(\vec{U}) \nabla_\vec{u} \Psi(\vec{u}^j,\vec{y},\vec{u}^{-j}) \\
		\Psi(\vec{u}^j,\vec{y},\vec{u}^{-j}) &= \Phi(\vec{u}^j,\vec{y}) - \frac{J \beta}{2(J-1)} \| \vec{u}^j - \bar{\vec{u}} \|^2,
	\end{aligned}
\end{equation*}
where we denote $\vec{u}^{-j} = \{ \vec{u}^k \}_{\substack{k=1 \\ k\neq j}}^J$. Within this formulation we see that our modified dynamics  again adds a  regularization term.
Existence and uniqueness of solutions to~\eqref{eq:stableLinearEnKF} is straightforward since the right--hand side is locally Lipschitz in $\vec{u}^j$, thus local existence of a solution in the space $\mathcal{C}([0,T))$ holds for some $T>0$. We need to prove global existence, namely that the solution does not blow up in finite time, and this is guaranteed by Proposition~\ref{th:spread} below.

We define for each $j=1,\dots,J$
\begin{align} %\label{eq:spreadresidual}
	\vec{e}^j(t) &= \vec{u}^j(t) - \bar{\vec{u}}(t), \label{eq:spread} \\
	\vec{r}^j(t) &= \vec{u}^j(t) - \vec{u}^* \label{eq:residual}
\end{align}
the ensemble spread and the residual to a value $\vec{u}^*$, respectively. Proposition~\ref{th:spread} gives  sufficient conditions for the existence of a monotonic decay for the ensemble spread.

\begin{proposition} \label{th:spread}
	Let $\vec{u}^{j}(0)\in\R^d$, $j=1,\dots,J$, be an admissible initial condition of the dynamical system~\eqref{eq:stableLinearEnKF}. The quantity $\left\| \vec{e}^j(t) \right\|^2$ is decreasing in time, i.e.
	$\left\| \vec{e}^j(t) \right\|^2 \leq \left\| \vec{e}^j(0) \right\|^2$, for each $j=1,\dots,J$ and $t \geq 0$, provided that $\alpha<1$ and $\beta<\min_k \lambda_{\vec{G}^T\vec{\Gamma}\vec{G}}^k$, where $\lambda_{\vec{G}^T\vec{\Gamma}\vec{G}}^k$ denotes the eigenvalues of $\vec{G}^T\vec{\Gamma}\vec{G}$. In particular, if $\vec{G}^T\vecsym{\Gamma}\vec{G}$ is positive definite then $\lim_{t\to\infty} \left\| \vec{e}^j(t) \right\|^2 = 0$.
\end{proposition}
\begin{proof}
	To prove the statement, we proceed similarly to existing theory.
	%it is sufficient to study the behavior of the covariance operator $\vec{C}$ in time.
	Let us denote $\vec{I}_d\in\R^{d\times d}$ the identity matrix. The hypothesis $\alpha<1$ implies $\tilde{\vec{C}}$ positive definite for all $t\geq 0$.
	%, and in particular $\vec{C} < \tilde{\vec{C}}$. Using $\beta<0$ and the entry--wise matrix norm,
	We compute
	%	\begin{align*}
	%	\frac{\mathrm{d}}{\mathrm{d}t} \| \vec{C} \| < -2 \| \langle \vec{C},\vec{C} \rangle_{\vec{P}_{\beta}} \| \leq 0
	%	\end{align*}
	\begin{align*}
		\frac12 \frac{\mathrm{d}}{\mathrm{d}t} \frac{1}{J} \sum_{j=1}^J \| \vec{e}^j(t) \|^2 = -\frac{1}{J} \sum_{j=1}^J \langle \vec{e}^j(t),\tilde{\vec{C}}(t) \vec{P}_\beta(t) \vec{e}^j(t) \rangle \leq 0
	\end{align*}
	where $\vec{P}_\beta = \vec{G}^T \vecsym{\Gamma} \vec{G} - \beta \vec{I}_d$. Sufficient condition for the last inequality being strict is $\vec{G}^T\vecsym{\Gamma}\vec{G}$ positive definite and $\beta<\min_k \lambda_{\vec{G}^T\vec{\Gamma}\vec{G}}^k$.
	%	In particular we have the following bound
	%	$$
	%	\vec{C}(t) = \left( 2\vec{P}_\beta t + \vec{C}(0)^{-1} \right)^{-1},
	%	$$
	%	and the velocity of decay is determined by the minimum eigenvalue of $\vec{P}_\beta$.
\end{proof}

The previous result establishes sufficient conditions for the ensemble collapse to the mean $\bar{\vec{u}}$ in the long time behavior, and consequently each ensemble member solves at equilibrium the same minimization problem that $\bar{\vec{u}}$ is solving. With this consideration we state and prove the following result on the convergence of the residual in the control space.

\begin{proposition} \label{th:residual}
	Let $\vec{u}^{j}(0)\in\R^d$, $j=1,\dots,J$, be an admissible initial condition of the dynamical system~\eqref{eq:stableLinearEnKF}. Assume that $\vec{G}^T \vecsym{\Gamma} \vec{G}$ is positive definite and let $\vec{u}^*$ be a KKT point of the minimization problem $\min_{\vec{u}\in\R^d} \Phi(\vec{u},\vec{y})$. Then $\lim_{t\to\infty} \left\| \vec{r}^j(t) \right\|^2 = 0$, for each $j=1,\dots,J$, provided that $\alpha<1$ and $\beta<\min_k \lambda_{\vec{G}^T\vec{\Gamma}\vec{G}}^k$, where $\lambda_{\vec{G}^T\vec{\Gamma}\vec{G}}^k$ denotes the eigenvalues of $\vec{G}^T\vec{\Gamma}\vec{G}$.
\end{proposition}
\begin{proof}
	By assumption $\vec{G}^T \vecsym{\Gamma} \vec{G}$ is positive definite and thus we have a unique global minimizer $\vec{u}^*$ of the minimization problem $\min_{\vec{u}\in\R^d} \Phi(\vec{u},\vec{y})$, for a given $\vec{y}\in\R^K$. Moreover, for Proposition~\eqref{th:spread} it is sufficient to show that $\|\bar{\vec{u}} - \vec{u}^* \| \to 0$ as $t\to\infty$. The evolution equation of the ensemble mean is given by
	$$
	\frac{\mathrm{d}}{\mathrm{d}t} \bar{\vec{u}} = - \tilde{\vec{C}} \nabla_{\vec{u}} \Phi(\bar{\vec{u}},\vec{y}).
	$$
	Then, since $\tilde{\vec{C}}$ is positive definite, at equilibrium the ensemble mean solves the equation $\nabla_{\vec{u}} \Phi(\bar{\vec{u}},\vec{y}) = 0$.
\end{proof}
%
%\begin{remark}
%	Contrary to the analysis in Section~\ref{ssec:stabilitySEnKF}, where the global asymptotic stability of the dynamical system~\eqref{eq:stableMoments} is guaranteed by choosing $\beta<1$, we observe that Proposition~\ref{th:spread} and Proposition~\ref{th:residual} assume \revision{$\beta\leq0$}. However, we stress the fact that this is a sufficient condition to show collapse of the ensemble and convergence to the residual. In order to strengthen this condition, and consider also values of \revision{$\beta\in(0,1)$}, one would need to ensure that the operator $\vec{P}_\beta$ is positive (semi)--definite.\todo[inline]{However, sufficient condition is $0<\beta<\min_k \lambda_{\vec{G}^T\vec{\Gamma}\vec{G}}^k$, where $\lambda_{\vec{G}^T\vec{\Gamma}\vec{G}}^k$ is the $k$-th eigenvalue of $\vec{G}^T\vec{\Gamma}\vec{G}$. Do you agree?}
%\end{remark}
%

\section{Numerical Simulations} \label{sec:numerics}

The simulations performed in this section are all obtained by the numerical solution of the ODE systems for the moments and for the ensemble dynamics.
For the sake of simplicity, and also to show that no suitable robust discretization is needed, we straightforwardly employ a first order explicit time integration with a fixed and small time step.
We observe that the mean--field limit would also allow to use a fast stochastic particle scheme, e.g.~the mean--field interaction algorithm, see~\cite{HertyVisconti2019} for application to the EKI, inspired on Direct Simulation Monte Carlo (DSMC) methods for kinetic equations.

\subsection{Simulation of the Moment Dynamics}

We recall that the stabilization of the continuous--time limit of the ensemble Kalman inversion is motivated by a linear stability analysis of the moment equations, in the simplest case of a one--dimensional control. For this reason, we aim to compare the moment dynamics provided by the ensemble Kalman inversion~\eqref{eq:continuousEnKF1} and by the present stabilization of the method~\eqref{eq:stableEnKF}. %We observe that the continuous--time limit formulation of~\eqref{eq:continuousEnKF1} can be recovered using~\eqref{eq:stableEnKF} with $\alpha=1$ and $\beta=0$.

All simulations run with the same parameters used for the stability analysis in Section~\ref{ssec:stabilitySEnKF}, namely we consider $G=\Gamma=1$. The stabilization parameter is $\alpha=0.1$ and the acceleration parameter is $\beta=-1$. Moreover, we set $y=2$ so that the target equilibrium is $F_y = (2,4)$.

\begin{figure}[t!]
	\centering
	\begin{subfigure}[t]{0.49\textwidth}
		\centering
		\includegraphics[width=\textwidth]{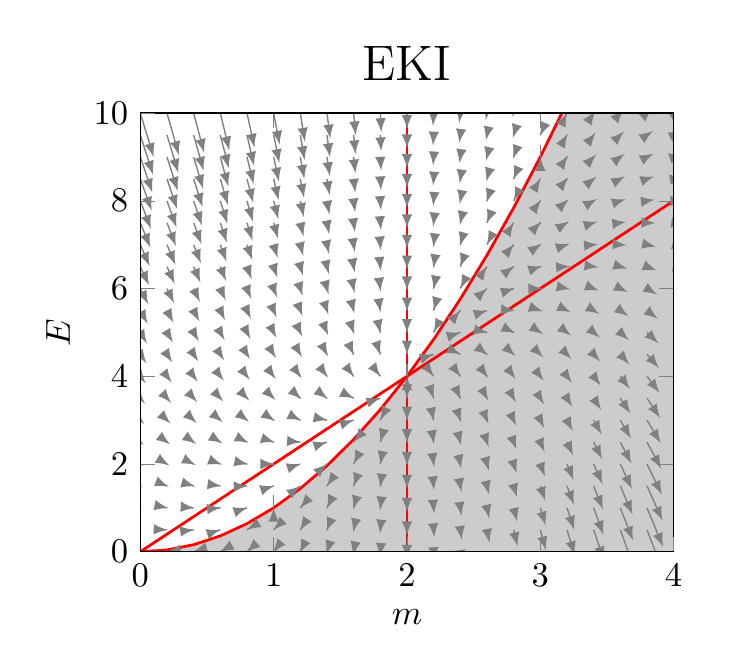}
		\caption{The classical ensemble Kalman inversion~\eqref{eq:momentODEs}.\label{fig:phasePlaneA}}
	\end{subfigure}
	\begin{subfigure}[t]{0.49\textwidth}
		\centering
		\includegraphics[width=\textwidth]{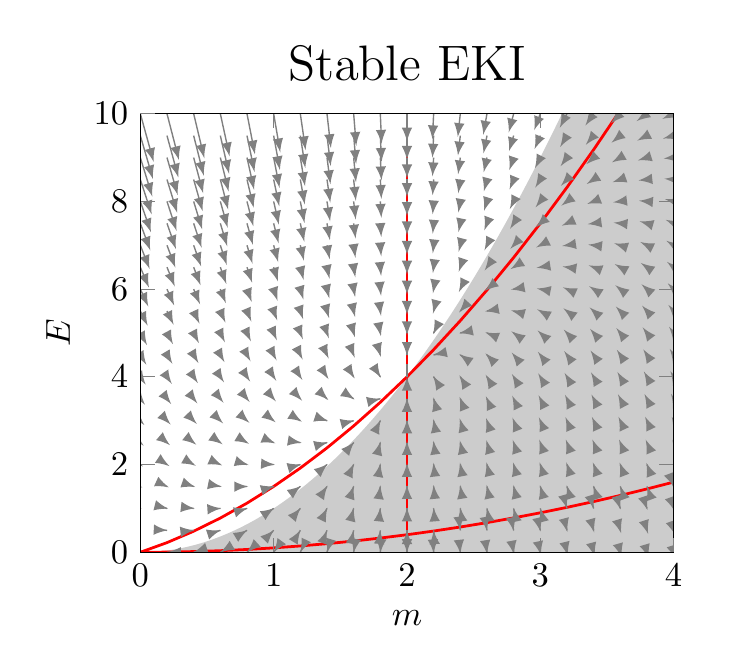}
		\caption{The stabilized ensemble Kalman inversion~\eqref{eq:stableMoments}.\label{fig:phasePlaneB}}
	\end{subfigure}
	\caption{Phase planes of the moment systems. Red lines are nullclines, the gray--shaded area represents the unfeasible region.\label{fig:phasePlane}}
\end{figure}

In Figure~\ref{fig:phasePlane} we show the phase portraits with the velocity field of the moment equations. The red lines are nullclines, and the gray--shaded area represents the unfeasible region where $E<m^2$. We observe that the stabilized version of the EKI proposed in this work preserves the target equilibrium $F_y$. In the classical EKI, Figure~\ref{fig:phasePlaneA}, the nullcline on the border of the feasible region is a set of equilibrium points. The stabilization moves these equilibria on the red nullcline in the unfeasible region, see Figure~\ref{fig:phasePlaneB}.

To provide additional numerical insight, we study the effect of the stabilization on the moment equations  by looking at the time behavior of the covariance $C(t)=E-m^2$ for two different initial conditions in both systems, the limit of the classical EKI and the stabilized version. In addition, for a thorough comparison, we take into account also the dynamics \rrevision{with $\beta=0$ in order to highlight the effect of the acceleration term.} 

\begin{figure}[t!]
	\centering
	\includegraphics[width=\textwidth]{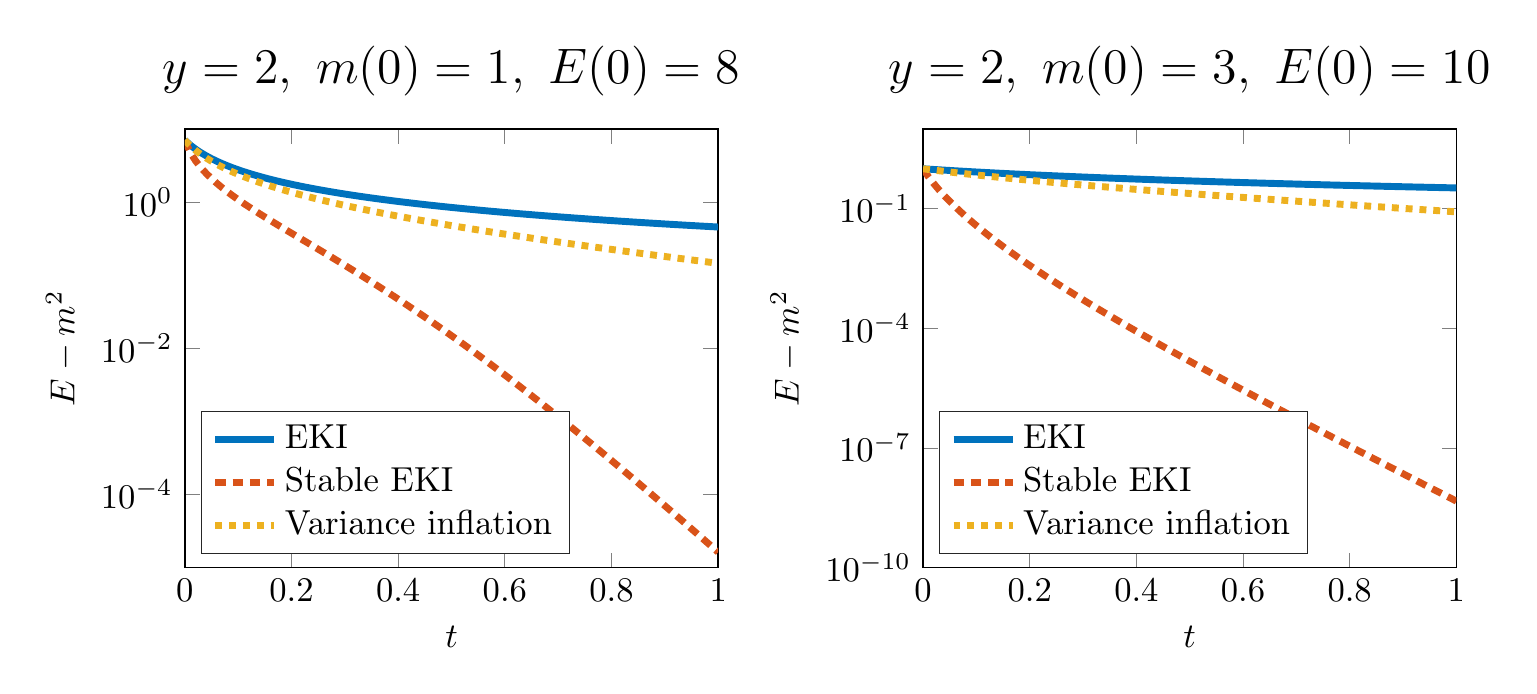}
	\caption{Variance evolution computed by solving the moment systems for the limit of the classical ensemble Kalman inversion~\eqref{eq:momentODEs}, the stabilized ensemble Kalman inversion~\eqref{eq:stableMoments}, and for the variance inflation only, i.e.~with $\beta=0$. We consider two
		different sets of initial conditions.\label{fig:variance}}
\end{figure}

The initial conditions of the first moments are $m(0)=1$ and $m(0)=3$. The initial energy $E(0)$ is 
chosen such that  $(m(0),E(0))$ is in the feasible region $C(0)\geq 0$. In Figure~\ref{fig:variance}, we observe that all the methods show a variance decay to zero, i.e.~collapse to a Dirac delta at mean--field level. However, noting the logarithmic scale in Figure~\ref{fig:variance} we see that the stable methods, characterized by variance inflation, decay to the equilibrium state much faster, precisely exponentially, than the limit of the classical EKI. Furthermore, we emphasize the role of the acceleration parameter $\beta$, which is taken into account only in the stable EKI method and which further speeds up the convergence speed. %In contrast, the stabilization obtained adding diffusion to the discrete dynamics modifies the equilibrium value of the energy and the variance does not decay to zero.

\subsection{A Two--Dimensional Inverse Problem}

We consider the inverse problem of finding the hydraulic conductivity function of a non--linear elliptic equation in two spatial dimension assuming that noisy observation of the solution to the problem are available.

The problem is described by the following PDE modeling groundwater flow in a two--dimensional confined aquifer:
\begin{equation} \label{eq:groundwater}
	\begin{aligned}
		-\nabla\cdot\left( e^{\log K} \nabla p \right) &= f \quad \mbox{ in } \ \Omega=(-1,1)^2 \\
		p &= 0 \quad \mbox{ on } \ \partial\Omega.
	\end{aligned}
\end{equation}
Here, $K$ is the hydraulic conductivity, $f$ is the force function and the flow is described in terms of the piezometric head $p$. This problem has been intensively used in the literature on the ensemble Kalman inversion to study performance of the method. E.g.~see~\cite{iglesiaslawstuart2013,ChadaStuartTong2019,schillingsstuart2017}.

We aim to find the log conductivity $u=\log K$ from $400$ observations of the solution $p$ on a uniform grid in $\Omega$. We choose $f = 100$. The mapping from $u$ to these observations is now non--linear, and thus we need to employ the ensemble dynamics for the non--linear model~\eqref{eq:stableEnKF}.

Noise is assumed to be Gaussian distributed with covariance $\mathbf{\Gamma}^{-1}=\gamma^2 \mathbf{I}$, with $\gamma=4$. The prior is also Gaussian distributed with covariance $(-\Delta)^{-2}$, whose discretization is again computed by using homogeneous Dirichlet boundary conditions. We use a $\mathbb{P}^1$ FEM approximation. The ensemble size is chosen as $J=100$. The ensemble dynamics~\eqref{eq:stableEnKF} are numerically solved by explicit Euler discretization with fixed and, to avoid stability issues, small time step $\Delta t=10^{-3}$. 

Final time for the simulations is determined by a stopping criterion
in order to avoid over--fitting of the method. We employ the discrepancy principle as stopping criterion. Thus, we check and stop the simulation when the condition $\vartheta \leq \| \vecsym{\eta} \|^2$ is satisfied, where $\vecsym{\eta}$ is the measurement noise and
\begin{equation} \label{eq:misfit}
	\vartheta = \frac{1}{J} \sum_{j=1}^J \| \G(\vec{u}^j) - \vec{p} - \vecsym{\eta} \|^2
\end{equation}
is the misfit which allows to measure the quality of the solution at each iteration. Moreover, $\vec{u}^j$ and $\vec{p}$ are vectors containing the discrete values of the control for the $j$--th ensemble member and of the true observations, respectively. In this example $\G$ is the $\mathbb{P}^1$ FEM discretization of the continuous operator defining the elliptic PDE~\eqref{eq:groundwater}.

The initial ensemble is drawn from a Gaussian distribution with given covariance matrix $\delta(-\Delta)^{-2}$, and we consider $\delta=1$ and $\delta=10^{-2}$. We compare results obtained with the continuous--time limit of the classical EKI, i.e.~when $\alpha=1$ and $\beta=0$, and with the stabilized method, using \rrevision{$\alpha=0.1$ and $\beta=-10$ when $\delta=1$, and $\alpha=0.9$ and $\beta=-0.1$ when $\delta=10^{-2}$. The inflation of the covariance is performed with $\vecsym{\Sigma}=\bar{\vecsym{\Sigma}}\bar{\vecsym{\Sigma}}^T$ where $\bar{\vecsym{\Sigma}}$ is full-rank matrix which contains pseudorandom values drawn from the standard normal distribution.}

\begin{figure}[!t]
	\centering
	\includegraphics[width=\textwidth]{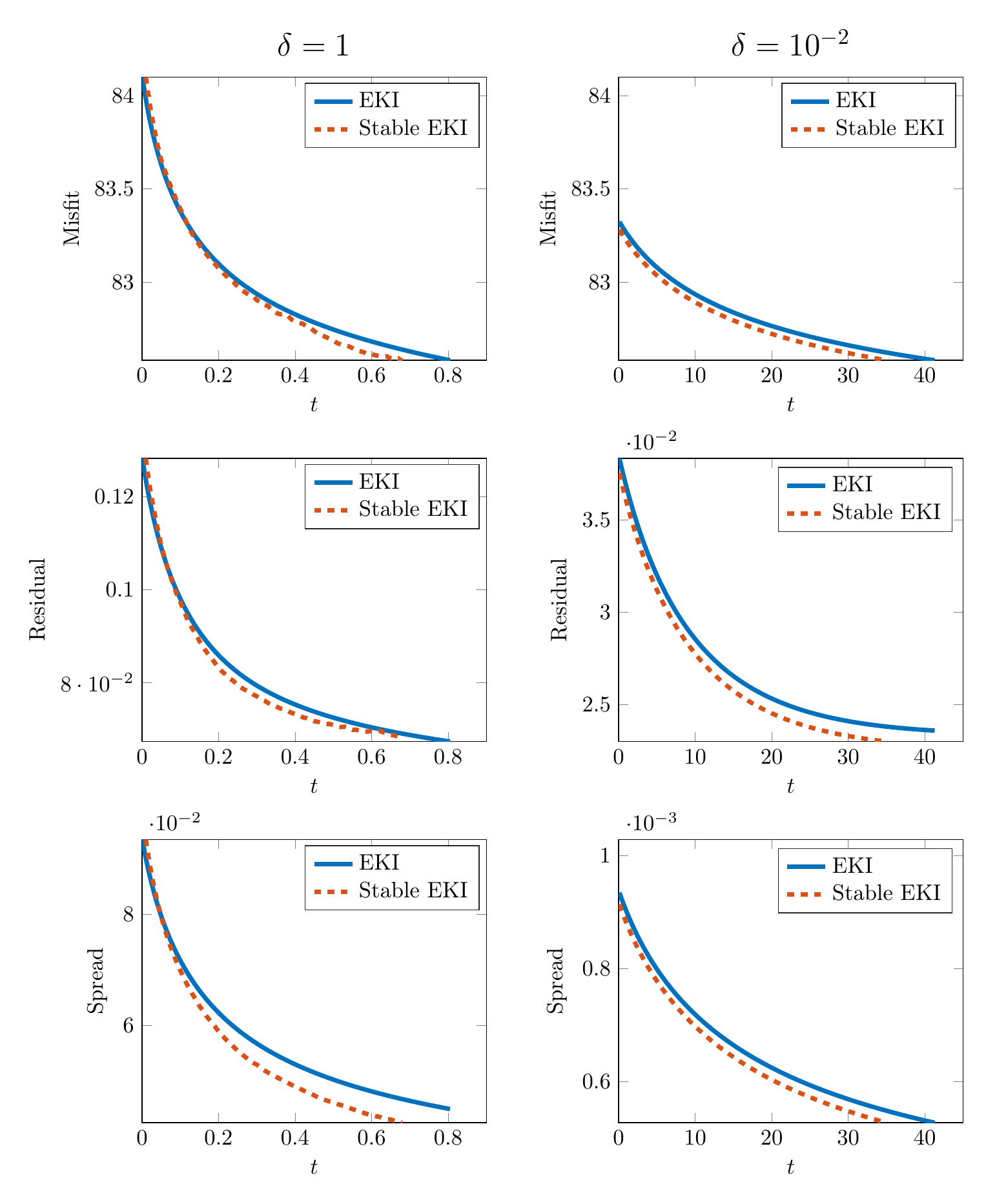}
	\caption{Misfit~\eqref{eq:misfit}, residual~\eqref{eq:residual} and spread~\eqref{eq:spread} behavior in time for the inverse problem of determining the log conductivity 
		$u=\log K$ for~\eqref{eq:groundwater} using the classical ensemble Kalman inversion (EKI)~\eqref{eq:continuousEnKF1} and the stabilized ensemble Kalman inversion (Stable EKI)~\eqref{eq:stableEnKF}. Left column: both methods converge for well chosen initial covariance ($\delta=1$). \rrevision{Right column represents the case of the overly confident prior, with $\delta=10^{-2}$, up to the time when the discrepancy principle is met}.\label{fig:analysis}}
\end{figure}

%\begin{figure}[!t]
%	\centering
%	\includegraphics[width=\textwidth]{figures/tikzSIAM/resultsSSsmallCovariance}
%	\caption{Inverse problem of determining the log conductivity $u=\log K$ for the two--dimensional groundwater equation~\eqref{eq:groundwater} on a $20\times 20$ grid, and solved by the continuous--time limit of the classical ensemble Kalman inversion~\eqref{eq:continuousEnKF1}. From top right: discrete observations of the true solution $p$; true observations perturbed by Gaussian noise; discrete true log conductivity $u$; solution computed with the identified unknown; one--dimensional plot of the discrete true and reconstructed log conductivity; discrete reconstructed log conductivity.\label{fig:resultsSSsmallCovariance}}
%\end{figure}

\begin{figure}[!t]
	\centering
	\includegraphics[width=\textwidth]{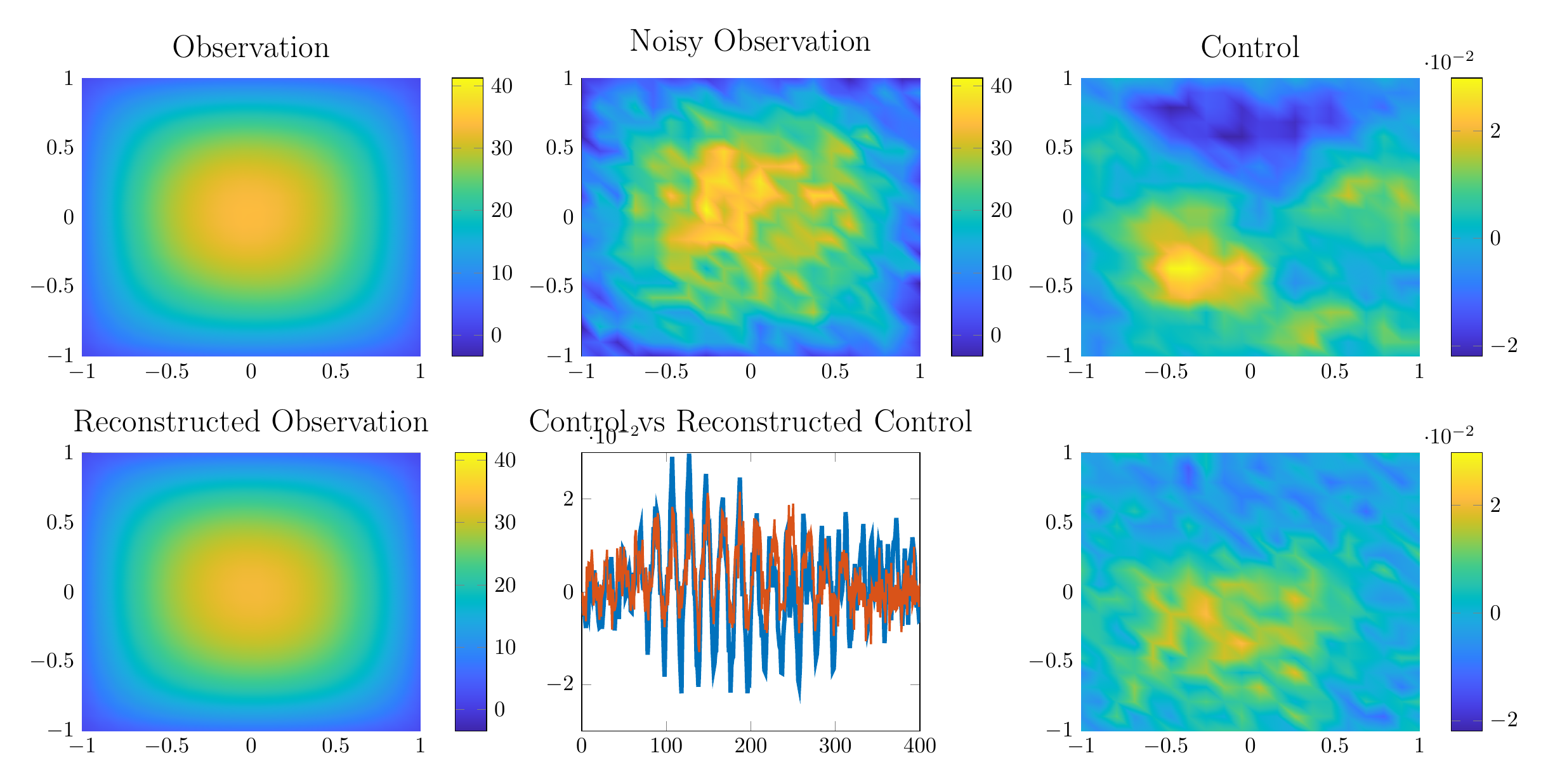}
	\caption{Inverse problem of determining the log conductivity $u=\log K$ for the two--dimensional groundwater equation~\eqref{eq:groundwater} on a $20\times 20$ grid, and solved by the stabilized ensemble Kalman inversion method~\eqref{eq:stableEnKF} with $\alpha=0.1$ and $\beta=-10$. From top right: discrete observations of the true solution $p$; true observations perturbed by Gaussian noise; discrete true log conductivity $u$; solution computed with the identified unknown; one--dimensional plot of the discrete true and reconstructed log conductivity; discrete reconstructed log conductivity\label{fig:resultsStableLargeCovariance}}
\end{figure}

In Figure~\ref{fig:analysis} we show the time behavior of the misfit~\eqref{eq:misfit} (top row), of the residual~\eqref{eq:residual} (middle row) and of the spread to the mean~\eqref{eq:spread} (bottom row) provided by the two methods. The results in the left panels are obtained with $\delta=1$, so that the initial ensemble is sampled from the same prior distribution of the exact control. \revision{The results in the right panels are computed with $\delta=10^{-2}$, which mimics the situation of an overly confident prior, and hence with the covariance close to the border of the feasible region, at two different final times.}
%More precisely, in the simulation depicted in the center column we allow the methods running until they meet the stopping criterion. Instead, in the right column we stop the simulations at time $T_\text{fin}=1$, as the Bayesian interpretation of the method would suggest.

We observe that, if the distribution of the initial ensemble is properly chosen, i.e.~when $\delta=1$, the two methods meet the discrepancy principle \rrevision{at time $T_\text{fin}=1$}, and the misfit, the residual and the ensemble spread \revision{monotonically} decrease in time. \revision{However, while the limit of the classical EKI meets the stopping criterion at time $t\approx 0.8$, the stabilized version of the method stops at a time $t\approx 0.68$. Therefore, in this example the stabilization allows us to save about 15\% of the computational cost.}
The difference between the two methods can be further appreciated when the initial guess of the ensemble is not properly chosen, i.e.~when $\delta=10^{-2}$. This is relevant in applications, where the distribution of the unknown control is not known, therefore the ensemble cannot be suitably initialized leading to a possible change in the length of the transient.
%We observe that the stabilized version of the method provides a fast transient of the misfit and of the residual \revision{when the final time is fixed to $T_\text{fin}=1$.} In contrast, the classical EKI which does not seem to converge.
\rrevision{If we stop the simulation at $T_\text{fin}=1$, as the Bayesian perspective suggests, both methods do not meet the discrepancy principle and one needs to run them for a longer time. Also in this example the stabilization allows us to save about 15\% of the computational cost.}

The effect on the performance of the methods when considering \rrevision{$\delta=1$ can be also observed in Figure~\ref{fig:resultsStableLargeCovariance} for the stabilized version of the method.} The top--row panels show, from left to right, the true solution $p$ of~\eqref{eq:groundwater} evaluated on a $20\times 20$ uniform grid, the perturbed solution by additive Gaussian noise, and the a--priori artificially assigned true log conductivity $u$, i.e.~the control in this example, which provides the solution $p$ and we aim to identify. The bottom--row panels, instead, show the solution obtained with the reconstructed log conductivity, and the identified control itself using both a one--dimensional and a two--dimensional visualization. In these figures we appreciate the \rrevision{good performance of the stabilized method which} is able to provide a good identification of the unknown control.

\section{Discussion and conclusions} \label{sec:conclusion}

\revision{The important point of this manuscript is the observation that the EKI leads to structurally unstable dynamical moment systems. This has technical aspects, like the fact that the target equilibrium is unstable from the unfeasible side of the phase space and that it is, while stable, approached only very slowly in time. However, this is also conceptually very important: Typically mathematical models that lead to structurally unstable dynamical systems are flawed -- the modeling part missed important aspects making the analysis of the resulting dynamical system highly susceptible to small noise and small variations in the model. A classical example is provided by chemical reactor modeling, see~\cite{Serra99}.}

\revision{Conventional wisdom is that the modeling leading to the structurally unstable system should be re-examined and it should be determined whether there are legitimate reasons (like in the case of Hamiltonian systems or systems with inherent symmetries) for the structural instability or whether the modeling process lead to the structurally unstable result.}

\revision{Due to the appeal of the EKI it is important to understand where the conceptual issue of a structurally unstable dynamical systems comes from. We have done that in this paper and shown that by computing the continuum limit of the classical EKI, terms are lost that lead to the structural instability. We also show how to make the EKI structurally stable. The stabilization relies on a suitable inflation of the covariance operator which makes the target equilibrium globally asymptotically stable, and thus approached exponentially fast in time.} The numerical results illustrate that the stabilized method is able to provide fast convergence to the solution, independently of the choice of the distribution for the initial ensemble.

\section*{Acknowledgments}

The authors thank the Deutsche Forschungsgemeinschaft (DFG, German Re-
search Foundation) for the financial support through 20021702/GRK2326, 333849990/IRTG-2379,
HE5386/15,18-1,19-1,22-1,23-1 and under Germany’s Excellence Strategy EXC-2023 Internet of Produc-
tion 390621612. The funding through HIDSS-004 is acknowledged.

This work has been initiated during the workshop ``Stochastic dynamics for complex networks and
systems'' hosted at the University of Mannheim in 2019. The workshop was financially supported
by the DAAD exchange project PPP USA 2019 (project-id 57444394).

G.V.~is member of the ``National Group for Scientific Computation (GNCS-INDAM)'' and acknowledges support by MUR (Ministry of University and Research)
PRIN2017 project number 2017KKJP4X.

\bibliographystyle{plain}
\bibliography{references}
\end{document}